\documentclass[11pt,reqno]{amsart}
\usepackage{xifthen,setspace}
\usepackage{subfig}
\usepackage{pkgfile}
\newcommand{\bA}{\boldsymbol A}
\newcommand{\bSigma}{\boldsymbol \Sigma}
\newcommand{\bD}{\boldsymbol D}
\newcommand{\bZ}{\boldsymbol Z}
\newcommand{\ba}{\boldsymbol a}
\newcommand{\bc}{\boldsymbol c}
\newcommand{\br}{\boldsymbol r}
\newcommand{\bu}{\boldsymbol u}

\title[Degree Sequence of Random Permutation Graphs]{Degree Sequence of Random Permutation Graphs}

\author{Bhaswar B. Bhattacharya}
\address{Department of Statistics, Stanford University, California, USA, 
{\tt bhaswar@stanford.edu}}
   
\author{Sumit Mukherjee}
\address{Department of Statistics, Columbia University, New York, USA, {\tt sm3949@columbia.edu}}

\begin{document}

\spacing{1.125}

\begin{abstract}
In this paper we study the degree sequence of the permutation graph $G_{\pi_n}$ associated with a sequence $\pi_n\in S_n$ of random permutations.  Joint limiting distributions of the degrees are established using results from graph and permutation limit theories. In particular, for the uniform random permutation, the joint distribution of the degrees of the vertices labelled $\ceil{nr_1}, \ceil{nr_2}, \ldots, \ceil{nr_s}$ converges (after scaling by $n$) to independent random variables $D_1, D_2, \ldots, D_s$, where $D_i\sim \dU(r_i, 1-r_i)$, for $r_i\in [0,1]$ and $i\in \{1, 2, \ldots, s\}$. Moreover, the degree of the mid-vertex (the vertex labelled $n/2$) has a central limit theorem, and  the minimum degree converges to a Rayleigh distribution after appropriate scalings.  Finally, the limiting degree distribution of the permutation graph associated with a Mallows random permutation is determined, and interesting phase transitions are observed. Our results extend to other exponential measures on permutations.
\end{abstract}


\maketitle

\section{Introduction}

Let $[n]:=\{1,2,\cdots, n\}$, and $S_n$ denote the set of all permutations of $[n]$. For any permutation $\pi_n\in S_n$ associate a {\it permutation graph} $G_{\pi_n}=(V(G_{\pi_n}), E(G_{\pi_n}))$, where $V(G_{\pi_n})=[n]$ and there exists an edge $(i, j)$ if and only if $(i-j)(\pi_n(i)-\pi_n(j))<0$, that is, whenever $i, j$ determines an {\it inversion} in the permutation $\pi_n$. The permutation graphs associated with $\pi_n$ and $\pi_n^{-1}$ are isomorphic, and the adjacency matrix $\bD_{\pi_n}$ associated with the permutation graph $G_{\pi_n}$ is 
$$
\bD_{\pi_n}(i,j):=
\left\{
\begin{array}{cc}
1  &  \text{ if }(i-j)(\pi_n(i)-\pi_n(j))<0, \\
0  &  \text{ otherwise. }
\end{array}
\right.$$

For a permutation $\pi_n\in S_n$, define $\ell_n(a)$ to be the line segment with endpoints $(a,0)$ and $(\pi_n(a),1)$, for $a\in [n]$. The endpoints of these segments lie on the two parallel lines $y = 0$ and $y = 1$, and two segments have a non-empty intersection if and only if they correspond to an inversion in the permutation. Thus, the permutation graph $G_{\pi_n}$ is the intersection graph of the segments $\{\ell_n(a)\}_{a=1}^n$. Moreover, for every two parallel lines, and every finite set of line segments with endpoints on both lines, the intersection graph of the segments is a permutation graph in the case that the segment endpoints are all distinct. The associated permutation graph can be constructed as follows: arbitrarily number the segments on one of the two lines in consecutive order, and read off these numbers in the order that the segment endpoints appear on the other line.

Permutation graphs were introduced by Pnueli et al. \cite{even_identify} and Even et al. \cite{even}, and they showed that a graph is a permutation graph if and only if the graph and its complement are transitively orientable, that is, an assignment of directions to the edges of the graph such that whenever there exist directed edges $(x,y)$ and $(y,z)$, there must exist an edge $(x,z)$. They also gave a polynomial time procedure to find a transitive orientation when it is possible. Testing whether a given graph is a permutation graph can be done in linear time \cite{lineartime}. Permutation graphs are perfect graphs and, as a result, several NP-complete problems may be solved efficiently for permutation graphs. For this reason, permutation graphs have found applications in many applied problems, like channel routing, scheduling, and memory allocation \cite{golumbic}. Permutation graphs also have applications in comparative genomics and bioinformatics \cite{bafna}.
 
Permutation statistics of a randomly chosen permutation, like its cycle structure, inversions, descents, and fixed points, are widely studied and are of fundamental importance in the analysis of algorithms. For connections of random permutation statistics with determinental point processes refer to Borodin et al. \cite{borodin_diaconis_fulman}.  Recently, Acan and Pittel \cite{permutation_connected_components} studied when $\sigma (n,m)$, a permutation chosen uniformly at random among all permutations of $[n]$ with $m$ inversions, is indecomposable (refer to \cite{bona_permutation_book,analytic_combinatorics} and the references therein for more on indecomposable permutations). The probability $p(n, m)$ that $\sigma (n,m)$ is indecomposable, is same as the probability that the random permutation graph $G_{\sigma(n, m)}$ is connected. Acan and Pittel \cite{permutation_connected_components} showed that $p(n, m)$ has a phase transition from 0 to 1 at $m_n:=(6/\pi^2)n\log n$. They also studied the behavior of $G_{\sigma(n, m)}$ at the threshold.

\begin{figure*}[h]
\centering
\begin{minipage}[c]{1.0\textwidth}
\centering
\includegraphics[width=3.8in]
    {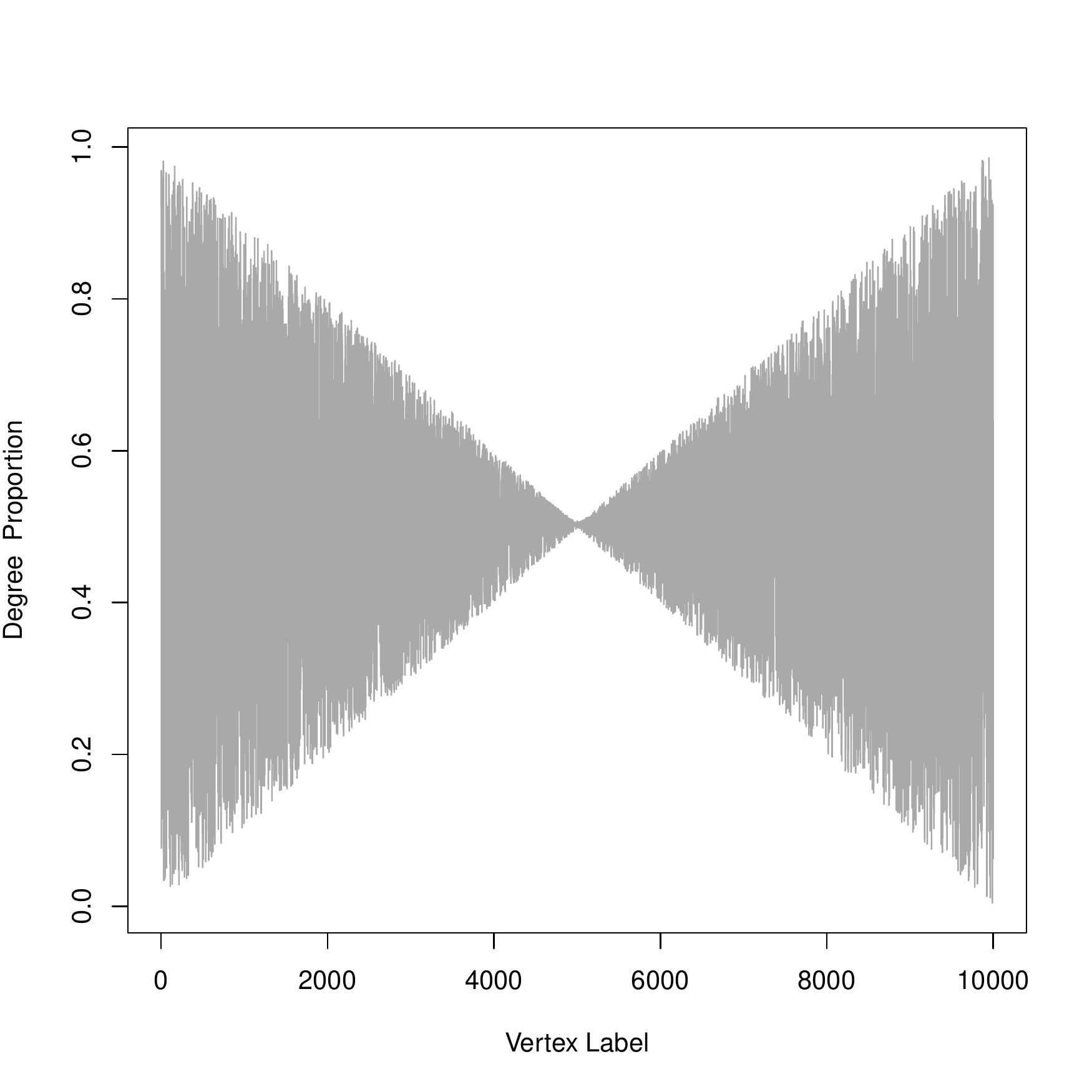}\\
\end{minipage}
\caption{\small{Degree distribution of labelled vertices of the permutation graph associated with a random permutation of length $n=10^5$.}}
\label{fig:degree_sequence}
\end{figure*}

One of the most important graph statistic is its degree sequence. Large graphs, like social network graphs, are often studied through  their degree sequence \cite{newman1,newman2}. For this reason, the degree sequence of random graphs has been widely studied over the years. Refer to Chatterjee et al. \cite{chatterjee_degree} and the references therein for results about random graphs with a given degree sequence. Degree distribution of a  random interval graph was studied by Scheinerman \cite{scheinerman}.  Diaconis et al. \cite{thresholdgraph,intervalgraph} studied the limits of threshold graphs and interval graphs, and initiated the limit theory of intersection graphs. They also considered the degree distribution of threshold graphs \cite{thresholdgraph}. 

In this paper, using results from the emerging literature on graph and permutation limit theories, we study the degree sequence of permutation graphs associated with a sequence $\pi_n\in S_n$ of random permutations. Given a permutation graph $G_{\pi_n}$ denote by $d_n(i):=\sum_{j=1}^n\bD_{\pi_n}(i,j)$ the degree of the vertex labelled $i\in [n]$. The quantity $d_n(i)/n$ will be referred to as the {\it degree proportion} of the vertex $i \in [n]$. Given a sequence $\pi_n\in S_n$ of random permutations, the {\it permutation process} is the stochastic process  on $(0,1]$ defined by \footnote{Throughout the paper, $\pi_n$ will be used interchangeably to denote both the permutation and the permutation process depending on the context. In particular, for $a\in [n]$ $\pi_n(a)$ will denote the image of $a$ under the permutation $\pi_n$. On the other hand, for $t \in [0, 1]$ $\pi_n(t)=\pi_n(\lceil nt\rceil)/n$ will denote the permutation process evaluated at $t$. Similarly, $d_n$ will be used to denote both the degree of a vertex and the degree process.}
\begin{equation}
\pi_n(t):=\frac{\pi_n(\ceil{nt})}{n}.
\label{permprocess}
\end{equation}
 It is shown in Theorem \ref{thm:degreejoint} that convergence in distribution of the permutation process implies joint convergence in distribution of permutations $\{\pi_n\}_{n \in \N}$ in the sense of Hoppen et al. \cite{permutation_limits} and the  {\it degree process} 
\begin{equation}
d_n(t):=\frac{d_n(\ceil{nt})}{n}. 
\label{degprocess}
\end{equation} 
The limiting degree process can be described in terms of the limit of the permutation process. As a consequence, we derive the degree distribution of the uniformly random permutation graph, that is, the permutation graph associated with a permutation chosen uniformly at random from $S_n$ (Corollary \ref{cor:degreejoint}). Figure \ref{fig:degree_sequence} shows the degree proportion of the vertices in the permutation graph associated with a uniformly random permutation of length $n=10^5$. It follows from Corollary \ref{cor:degreejoint} that the degree proportion of the mid-vertex, that is, the vertex labelled $\ceil{n/2}$, converges to 1/2 in probability (This can also be seen from the fan-like structure in Figure \ref{fig:degree_sequence} around the point 1/2). We show that $d_n(\ceil{n/2})/n$ has a CLT around 1/2 after an appropriate rescaling (Theorem \ref{th:degreenormal}). We also show that the scaled minimum degree in a uniformly random permutation graph converges to a Rayleigh distribution with parameter $\frac{1}{\sqrt 2}$ (Theorem \ref{min_deg}). 

Moreover, we give sufficient conditions for verifying the convergence of the permutation process. These conditions can be easily verified for many non-uniform (exponential) measures on permutations. These conditions  together with the recent work of Starr \cite{starr} can be used to explicitly determine the limiting distribution of the degree process for a Mallows random permutation, for all $\beta \in \R$ (Theorem \ref{th:mallowsdegreejoint}). For each $a\in (0, 1]$, the limiting density of $d_n(\ceil{na})/n$ has a interesting phase transition depending on the value of $\beta$: there exists a critical value $\beta_c(a)$ such that for $\beta \in [0, \beta_c(a)]$ the limiting density is a continuous function supported on $[a, 1-a]$. However, for  $\beta> \beta_c(a)$ the density breaks into two piecewise continuous parts. If $\beta=1/T$ denotes the inverse temperature, then this is the statistical physics phenomenon of {\it replica symmetry breaking}  in the low temperature regime.

The rest of the paper is organized as follows: Section \ref{sec:graph_permutation_limit} contains the basics of graph and permutation limit theories and their connections. Section \ref{sec:summary} gives the summary of our main results. The limiting distribution of the degree processes, relating it to the convergence of the permutation process, is in Section \ref{sec:main}. Section \ref{sec:uniform} covers uniformly random permutation graphs, including the CLT for the mid-vertex. Conditions to verify the convergence of the permutation process are given in Section \ref{sec:equic}. These conditions are also verified for a general class of exponential measures on permutations. The application of these results to the Mallows random permutation is given in Section \ref{sec:mallows}.  Section \ref{sec:min_deg} contains the asymptotics for the minimum degree of a uniformly random permutation graph.

\section{Graph and Permutation Limit Theories}
\label{sec:graph_permutation_limit}

\subsection{Graph Limit Theory}

The theory of graph limits was developed by Lovasz and coauthors \cite{graph_limits_I,graph_limits_II,lovasz_book}, and has received phenomenal attention over the last few years. Graph limit theory connects various topics such as graph homomorphisms, Szemer\'edi's regularity lemma, quasirandom graphs, graph testing and extremal graph theory, and has even found applications in statistics and related areas. For a detailed exposition of the theory of graph limits refer to Lovasz \cite{lovasz_book}. 

Here we mention the basic definitions about convergence of graph sequences. If $F$ and $G$ are two graphs, then define the homomorphism density of $F$ into $G$ by
$$t(F,G) :=\frac{|\hom(F,G)|}{|V (G)|^{|V (F)|},}$$
where  $|\hom(F,G)|$ denotes the number of homomorphisms of $F$ into $G$. In fact, $t(F, G)$ is the probability that a random mapping $\phi: V (F) \rightarrow V (G)$ defines a graph homomorphism. The basic definition is that a sequence $G_n$ of graphs converges if $t(F, G_n)$ converges for every graph $F$.

There is a natural limit object in the form of a function $W \in \sW$, where $\sW$ is the space of all measurable functions from $[0, 1]^2$ into $[0, 1]$ that satisfy $W(x, y) = W(y,x)$ for all $x, y$. Conversely, every such function arises as the limit of an appropriate graph sequence. This limit object determines all the limits of subgraph densities:
if $H$ is a simple graph with $V (H)= \{1, 2, \ldots, |V(H)|\}$, let
$$t(H,W) =\int_{[0,1]^{|V(H)|}}\prod_{(i,j)\in E(H)}W(x_i,x_j) dx_1dx_2\cdots dx_{|V(H)|}.$$
A sequence of graphs $\{G_n\}_{n\geq 1}$ is said to converge to $W$ if for every finite simple graph $H$, 
\begin{equation}
\lim_{n\rightarrow \infty}t(H, G_n) = t(H, W).
\label{eq:graph_limit}
\end{equation}
These limit objects, that is, elements of $\sW$, are called {\it graph limits} or {\it graphons}. A finite simple graph $G$ on $[n]$ can also be represented as a graphon in a natural way: Define $f^G(x, y) =\boldsymbol 1\{(\ceil{nx}, \ceil{ny})\in E(G)\}$, that is, partition $[0, 1]^2$ into $n^2$ squares of side length $1/n$, and define $f^G(x, y)=1$ in the $(i, j)$-th square if $(i, j)\in E(G)$ and 0 otherwise. Observe that $t(H, f^{G}) = t(H,G)$ for every simple
graph $H$ and therefore the constant sequence $G$ converges to the graph limit $f^G$. 

The notion of convergence in terms of subgraph densities outlined above can be
captured by the cut-distance defined as:
$$d_\square(f,g):=\sup_{S,T\subset [0,1]}\left|\int_{S\times T}[f(x,y)-g(x,y)]dxdy\right|,$$
for $f, g\in\sW$. Define an equivalence relation on $\sW$ as follows: $f\sim g$ whenever $f(x, y)=g_\sigma(x, y):=g(\sigma x, \sigma y)$, for some measure preserving bijection $\sigma:[0,1]\mapsto [0,1] $. Denote by $\widetilde g$ the closure of the orbit $g_\sigma$ in $(\sW, d_\square)$. The space $\{\widetilde g:g\in \sW\}$ of closed equivalence classes  is denoted by $\widetilde{\sW}$ and is associated with the following natural metric:
$$\delta_\square(\widetilde{f},\widetilde{g}):=\inf_{\sigma}d_\square(f,g_\sigma)=\inf_{\sigma}d_\square(f_\sigma,g)=\inf_{\sigma_1,\sigma_2}d(f_{\sigma_1},g_{\sigma_2}).$$
The space $(\widetilde{\sW},\delta_\square)$ is compact \cite{graph_limits_I}, and the  metric $\delta_\square$ is commonly referred to as the cut-metric. 

The main result in graph limit theory is that a sequence of graphs $\{G_n\}_{n\geq 1}$
converges to a limit $W\in \sW$ in the sense defined in (\ref{eq:graph_limit}) if and only if
$\delta_\square(\widetilde f^{G_n}, \widetilde W)\rightarrow 0$ \cite[Theorem 3.8]{graph_limits_I}. More generally, a sequence $\{\widetilde W_n\}_{n\geq 1}$ converges to $\widetilde W\in \widetilde \sW$ if and only if $\delta_\square(\widetilde W_n, \widetilde W)\rightarrow 0$.

\subsection{Permutation Limits}

Analogous to the theory of graph limits  Hoppen et al.  \cite{permutation_testing,permutation_limits} developed the theory of permutation limits.  For $\pi_n\in S_n$ and $\sigma \in S_a$,  $\sigma$ is a sub-permutation of $\pi_n$ if there exists $1 \leq  i_1 <\cdots < i_a\leq n$ such that such that $\sigma(x) < \sigma(y)$ if and only if $\pi_n(i_x)<\pi_n(i_y)$. For example, 132 is a sub-permutation of 7126354 induced by $i_1=3, i_2=4, i_3=6$. Sub-permutations are often referred to as {\it patterns} and their combinatorial properties are widely studied (refer to Bona \cite{bona_pattern_normality}, the recent paper of Janson et al. \cite{janson_multiple_pattern} and the references therein). We follow the setup of Hoppen et al.  \cite{permutation_limits} (see also K\'ral' and Pikhurko \cite{kral_pikhurko} and Glebov et al. \cite{permutons_graphons}), and refer to them as sub-permutations because of its similarity to the notion of sub-graphs, the counterpart object in graph limit theory. 

The density of a permutation $\sigma\in S_a$ in a permutation $\pi_n\in S_n$ is 
$$d(\sigma, \pi_n)=
\left\{
\begin{array}{ccc}
{{n}\choose{a}}^{-1}\#\{\sigma \in S_a: \sigma \text{ is sub-permutation of } \pi_n\}  &  \text{ if } & a \leq n \\ 
0  &     \text{ if } & a > n.
\end{array}
\right.
$$
An infinite sequence $(\pi_n)_{n\in \N}$ of permutations  is convergent as $n \rightarrow \infty$ if $d(\sigma, \pi_n)$ converges for every permutation $\sigma$. 

Every convergent sequence of permutations can be associated with an analytic object, referred to as  {\it permuton}, which is a probability measure $\nu$ on $([0, 1]^2, \sB([0, 1]^2)$ with uniform marginals, where $\sB([0, 1]^2)$ is the sigma-algebra of the Borel sets of $[0, 1]^2$. For an integer $n$, sample $n$ independent points $(x_1, y_1), (x_2, y_2), \ldots, (x_n, y_n)$ in $[0, 1]^2$ randomly from the measure $\nu$. Let $\sigma_x$ and $\sigma_y$ be the permutations of order $n$ such that $x_{\sigma_x(1)}< x_{\sigma_x(2)}<\cdots x_{\sigma_x(n)}$ and $y_{\sigma_y(1)}< y_{\sigma_y(2)}<\cdots y_{\sigma_y(n)}$, respectively. Define $\sigma_y^{-1}\circ\sigma_x$ as the $\nu$-{\it random permutation} of order $n$. For a permuton $\nu$ and a permutation $\sigma\in S_k$, define $d(\sigma, \nu)$ as the probability that a $\nu$-random permutation of order $k$ is $\sigma$. 
The main result in permutation limit theory \cite{permutation_limits} is that for every convergent sequence $(\pi_n)_{n \in \N}$ of permutations, there exists a unique permuton $\nu$ such that 
\begin{equation}
d(\sigma, \nu) = \lim_{n\rightarrow \infty} d(\sigma, \pi_n),
\label{eq:permutation_limit}
\end{equation}
for every permutation $\sigma$. The {\it permuton} $\nu$ is defined as the limit of the sequence $(\pi_n)_{n \in \N}$. On the other hand, a sequence of $\nu$-random permutations $(\pi_n)_{n \in \N}$ converges to $\nu$ in the limit.

As in graph limit theory, the above notion of permutation convergence can be metrized by 
embedding all permutations to the space probability measures $\cM$ on $[0, 1]^2$ with uniform marginals, equipped with any metric which induces the topology of weak convergence. To this end, define for any $\pi_n\in S_n$, a probability measure $\nu_{\pi_n}\in \cM$
as: 
\begin{equation}
d\nu_{\pi_n}:=f_{\pi_n}(x,y) dxdy, \text{ where } f_{\pi_n}(x,y)=n\pmb 1\{(x,y):\pi_n(\lfloor n x\rfloor)=\lfloor ny\rfloor\},
\label{eq:permuton}
\end{equation}
is the density of $\nu_{\pi_n}$ with respect to Lebesgue measure. Like in graph limit theory, $\nu_{\pi_n}$ has the following interpretation: Partition $[0, 1]^2$ into $n^2$ squares of side length $1/n$, and define $f_{\pi_n}(x, y)=n$ for all $(x, y)$ in the $(i, j)$-th square if $\pi_n(i)=j$ and 0 otherwise. The measure $\nu_{\pi_n}$ will be referred to as the {\it permuton associated with $\pi_n$}.


The space $\cM$ equipped with the topology of weak convergence of measures is compact. Moreover, the map $\pi\rightarrow \nu_{\pi}$ is 1-1 and so all finite permutations are contained in $\cM$. Analogous to the result in graph limit, Hoppen et al. \cite[Lemma 5.3]{permutation_limits} showed that a sequence of permutations $\pi_n$ convergences in the sense of (\ref{eq:permutation_limit}) if and only if the corresponding sequences of measures $\nu_{\pi_n}$ converges weakly to $\nu\in \cM$. More generally, any sequence of measures $\{\nu_n\}$ in $\cM$ converges weakly to a measure $\nu$ if and only if $d(\sigma,\nu_n)\rightarrow d(\sigma,\nu)$ for all permutations $\sigma$.

\subsection{Limit of Permutation Graphs} 

Diaconis et al. \cite{thresholdgraph,intervalgraph} studied the limits of threshold graphs and interval graphs. They also mentioned that their methods can be used for other classes of intersection graphs, which include permutation graphs \cite{intervalgraph}. However, instead of describing the limit object as symmetric function from $[0, 1]^2$ to $[0, 1]$, they 
represented the graph limit as a measure on $[0, 1]^2$ with uniform marginals. 

To this end, for every measure $\nu$ on $[0, 1]^2$ with uniform marginals, they defined a unique graph limit object $\widetilde{W}_\nu\in \tilde\sW$ by specifying $t(F, W_\nu)$ for all graphs $F$ as follows:
\begin{eqnarray}
t(F, W_\nu)&:=&\E\prod_{(i, j)\in E(F)}K(X_i, X_j)\nonumber\\
&=&\int_{[0,1]^{2|V(F)|}}\prod_{(i,j)\in E(F)}K(x_i,x_j) d\nu(x_1)d\nu(x_2)\cdots d\nu(x_{|V(F)|}),
\label{eq:permutation_graph_limit}
\end{eqnarray}
where $X_1, X_2, \ldots, X_n$ are independent and identically distributed from $\nu$ and $K: [0, 1]^2\times [0, 1]^2\rightarrow [0, 1]$ given by $K((a_1, b_1), (a_2, b_2))=\boldsymbol 1\{(a_1-a_2)(b_1-b_2)<0\}$.

In the graph limit literature it is usually convenient to represent a graphon by a functional 
$W : [0, 1]^2\rightarrow [0, 1]$. 
 However, Diaconis et al. \cite{intervalgraph} described permutation graph limits 
 in terms of a permuton $\nu$. In that case every probability measure $\nu$ on $[0, 1]^2$ defines a graph limit $W_\nu$ by (\ref{eq:permutation_graph_limit}). Diaconis et al. \cite{intervalgraph} showed that every permutation graph limit may be represented in terms of a measure $\nu$ on $([0, 1]^2, \sB([0, 1]^2))$ via \eqref{eq:permutation_graph_limit} with the two marginal distributions of $\nu$  both being uniform on $[0,1]$.

Glebov et al. \cite{permutons_graphons} pointed out that if $(\pi_n)_{n\in\N}$ is a convergent sequence of permutations (in the sense of Hoppen et al. \cite{permutation_limits} described above), then the sequence of permutation graphs $(G_{\pi_n})_{n\in \N}$ is also convergent. 
 Therefore, each permuton $\nu$  can be associated with an equivalence class $\widetilde{W_\nu}\in \widetilde \sW$. However, the map $\nu\rightarrow \widetilde{W_\nu}$ is not one-to-one, as can be seen by the following example:  Suppose $\nu\in\cM$ be a permuton which is not exchangeable, i.e. if $(X,Y)$ has distribution $\nu$ then $(Y,X)$ has a distribution $\mu\neq\nu$. If $\pi_n$ be a sequence converging to $\nu$, then $\pi_n^{-1}$ converges to $\mu$.  Since the graphs $G_{\pi_n}$ and $G_{\pi_n^{-1}}$ are isomorphic, $\widetilde W_\nu$ are $\widetilde W_\mu$ are identical, but by choice $\mu$ and $\nu$ are not same.
 
For other interesting connections between graph and permutation limit theories refer to recent papers of K\'ral' and Pikhurko \cite{kral_pikhurko}, Glebov et al. \cite{permutons_graphons,permutation_graph_testing}.

\subsection{Empirical Degree Distribution of Random Permutation Graphs}
Given a sequence of permutation graphs $(G_{\pi_n})_{n\geq 1}$, the empirical degree distribution is $$\kappa(G_{\pi_n}):=\frac{1}{n}\sum_{i=1}^n\delta_{\frac{d_n(i)}{n}}.$$ 
Diaconis et al. \cite{intervalgraph} pointed out that if $G_{\pi_n}\rightarrow W_\nu$, then $\kappa(G_{\pi_n})$ converges weakly to the distribution of the random variable
$$W_1(\vec X):=\int_{[0,1]^2}K(\vec X, \vec z) d\nu(\vec z),$$
where $\vec X=(X_1, Y_1)$ is a random element in $[0, 1]^2$ with distribution $\nu$, and $K((x_1, y_1), (x_2, y_2))=\boldsymbol 1\{(x_1-x_2)(y_1-y_2)<0\}$. 

The degree distribution of a  permutation graph is an immediate consequence of this discussion as explained below.

\begin{ppn}Let $(\pi_n)_{n\in \N}$ be a sequence of permutations converging to a permuton $\nu$ in $([0, 1]^2, \sB([0, 1]^2))$, and $G_{\pi_n}$ be the associated permutation graph. Then the empirical degree distribution
\begin{equation}\kappa(G_{\pi_n}):=\frac{1}{n}\sum_{i=1}^n\delta_{\frac{d_n(i)}{n}}\dto X_1+Y_1-2F_\nu(X_1, Y_1),
\label{eq:degree_nu}
\end{equation}
where $(X_1, Y_1)\sim \nu$ and $F_\nu$ is the distribution function of $\nu$.
\label{ppn:degree_nu}
\end{ppn}

\begin{proof}Let $K: [0, 1]^2\times [0, 1]^2\mapsto [0, 1]$ be $K((x_1, y_1), (x_2, y_2))=\boldsymbol 1\{(x_1-x_2)(y_1-y_2)<0\}$. From the discussion before the proposition, $\kappa(G_{\pi_n})$ converges weakly to the law of
$$W_1(X_1, Y_1)=\int_{[0,1]^2}\boldsymbol 1\{(X_1-x_2)(Y_1-y_2)<0\}d\nu(x_2, y_2)= \nu\left([0, X_1]\times [Y_1, 1]\right)+\nu\left([0, Y_1]\times [X_1, 1]\right).$$
The  limiting degree distribution in (\ref{eq:degree_nu}) now follows from the fact that $\nu$ has uniform marginals.
\end{proof}

\section{Summary of the Results}
\label{sec:summary}

Let $\pi_n\in S_n$ be any permutation. Define the random variable $a_n(i):=\sum_{j=1}^{i-1}\bD_{\pi_n}(i,j)$, for $i \in [n]$. Note that $a_n(i)\in [0, i-1]$ represents the number of edges in $G_{\pi_n}$ connecting the vertex $i$ to vertices $j\in [i-1]$. The quantity $a_n(i)$ will be referred to as the {\it back-degree} of the vertex $i$. Similarly, one can define the {\it front-degree} of the vertex $i$ as $b_n(i):=\sum_{j=i+1}^n\bD_{\pi_n}(i,j)$. Note that $b_n(i)\in [0,  n-i]$ and $d_n(i)=a_n(i)+b_n(i)$, where $d_n(i)$ is the degree of the vertex $i$. We are interested in the convergence of the {\it degree process} 
\begin{equation}
d_n(t)=d_n(\ceil{nt})/n=a_n(t)+b_n(t),
\label{eq:deg}
\end{equation}
where $a_n(t)=a_n(\ceil{nt})/n$ and $b_n(t)=b_n(\ceil{nt})/n$, are the {\it back-degree process} and  the {\it front-degree process}, respectively. 

Since $[0,1]^{(0,1]}$ is compact by Tychonoff's theorem, a sequence of stochastic processes $\{Z_n(t)\}_{t\in (0,1]}$ supported on $[0,1]$ converges in law to a process $\{Z(t)\}_{t\in (0,1]}$ (denoted by $Z_n(\cdot)\stackrel{w}{\Rightarrow}Z(\cdot)$), if and only if the finite dimensional distributions of $Z_n(\cdot)$ converge to $Z(\cdot)$. A sequence of random measures $\mu_n\in \cM$ converges {\it weakly in distribution} to a random measure $\mu\in \cM$ (denoted by $\mu_n\dto \mu$), if 
\begin{equation}
\mu_n(f):=\int_{[0, 1]^2}fd\mu_n\dto\int_{[0, 1]^2}fd\mu:=\mu(f),
\end{equation}
for all continuous functions $f: [0, 1]^2\rightarrow \R$. Finally, we denote by $\sL(X)$, the distribution of the random variable $X$.

Unlike the empirical degree distribution, the convergence of a permuton sequence $(\nu_{\pi_n})_{n\geq 1}$ (defined in (\ref{eq:permuton})), might not imply convergence of the degree process. An easy way to see this is to consider $\pi_n$  a sequence of deterministic permutations converging to $\lambda([0, 1]^2)$, the Lebesgue measure on the unit square, and $\sigma_n$ a uniformly random permutation on $S_n$. Define a new sequence of permutations $\kappa_n$ which is  $\pi_n$ if $n$ is odd and $\sigma_n$ if $n$ is even. Clearly, $\nu_{\kappa_n}$ converges to $\lambda([0, 1]^2)$. However, for $t\in [0, 1]$, the degree process $d_n(t)$ converges in distribution to the random variable  $\dU[t, 1-t]$ along $n$ even, and is a deterministic sequence of real numbers for $n$ odd.

The above discussion necessitates extra assumptions on the random permutations to ensure the convergence of the degree process. To this end, for $\pi_n\in S_n$ and $t\in (0,1]$, define 
\begin{equation}
\pi_n(t):=\pi_n(\lceil nt\rceil)/n.
\label{eq:perm}
\end{equation}
By this rescaling,  $\pi_n:(0,1]\mapsto [0,1]$ is a stochastic process, hereafter referred to as the {\it permutation process}. The following theorem shows that the convergence of the permutation process implies that both the corresponding permutons and the degree process converge. To best of our knowledge, this connection between the convergence of the permutation process and the permutons is new, and might be of independent interest. In fact, under regularity conditions (discussed in Section \ref{sec:equic}) the two notions of convergence are equivalent. 

\begin{thm}
Let $\pi_n\in S_n$ be a sequence of random permutations such that  
\begin{eqnarray}\pi_n(\cdot)\stackrel{w}{\Rightarrow}Z(\cdot).
\label{eq:convprocessthm}
\end{eqnarray}
Then  there exists a (random) measure $\mu \in \cM$ such that the permuton $\nu_{\pi_n}\dto \mu$, and the degree process
\begin{equation}
d_n(\cdot)\stackrel{w}{\Rightarrow} D(\cdot),
\label{eq:deg_conv}
\end{equation} 
where $$D(t)=t+Z(t)-2F_\mu(t, Z(t)),$$ and $F_\mu$ is the distribution function of the measure $\mu$.
\label{thm:degreejoint}
\end{thm}

The above theorem, which is proved in Section \ref{sec:thmain}, will be used to determine the limiting degree process for various random permutations. Note that the limiting measure $\mu$ might be random. In this case, the finite dimensional distributions degree process can be dependent (see Example \ref{dependent}). However, for most of the examples considered in this paper the limiting measure in non-random, and the corresponding degree process has independent finite dimensional distributions. This is summarized in the following corollary and proved in Section \ref{sec:cormain}.

\begin{cor}
Suppose the permutation process $\pi_n(\cdot) \stackrel{\sD}\Rightarrow  Z(\cdot)$, and the finite dimensional marginals of $Z(\cdot)$ are independent. 

\begin{itemize}
\item Then $\nu_{\pi_n}$ converges to a non random measure $\mu\in \cM$  with law of $(X,Y)\sim\mu$ as follows: $$X\sim \dU[0,1], \text{ and } \sL(Y|X=x)\sim \sL(Z(x)).$$  

\item The finite dimensional marginals of the limiting degree process $D(\cdot)$ are independent.
\end{itemize}
\label{cor:independent}
\end{cor}

The above results imply that to determine the convergence of the degree process, it suffices to verify the convergence of the permutation process. This requires the permutation process to have some regularity, which is formalized in Section \ref{sec:equic}. The regularity conditions ensure that the finite dimensional distributions of the permutation process are ``equicontinuous",  which can be verified easily for exponential models on permutation, like the Mallow's model \cite{clmallows} and Spearman's rank correlation models \cite{diaconis_group}.

\subsection{Uniform Random Permutation}

A {\it uniformly random permutation graph} $G_{\pi_n}$ is the permutation graph associated with a uniformly random permutation $\pi_n\in S_n$.  The degree distribution for a uniformly random permutation graph can derived easily from Theorem \ref{thm:degreejoint}. To this end, denote by $[a, b]$, for $a, b \in \R$, the interval $[a\wedge b, a\vee b]$, where $a\wedge b:=\min\{a, b\}$ and $a\vee b:=\max\{a, b\}$. Moreover, $\dU(a, b)$, for $a, b\in [0, 1]$,  denotes the uniform distribution over $[a, b]$.

\begin{cor}Let $\pi_n\in S_n$ be a uniform random permutation and $d_n(1), d_n(2), \ldots d_n(n)$ be the degree sequence of the associated permutation graph $G_{\pi_n}$. Then for indices $0\leq  r_1<r_2< \cdots <r_s\leq 1$,
$$\left(\frac{d_n(\ceil{nr_1})}{n}, \frac{d_n(\ceil{nr_2})}{n}, \ldots, \frac{d_n(\ceil{nr_s})}{n}\right) \stackrel{\sP}{\rightarrow}(D_1, D_2, \ldots, D_s),$$
where $D_1, D_2, \ldots, D_s$ are independent and $D_i\sim \dU(r_i, 1-r_i)$, for every $i \in [s]$.
\label{cor:degreejoint}
\end{cor}

\begin{proof}For a uniformly random permutation, the permutation process $\pi_n(\cdot)\stackrel{w}\Rightarrow Z(\cdot)$, where $Z(t)$ is independent $\dU[0, 1]$, for all $t\in [0, 1]$. Moreover, the permuton $\nu_{\pi_n}$ converges to the Lebsegue measure on $[0, 1]^2$. Therefore, for $a\in [0, 1]$, by Theorem \ref{thm:degreejoint} $$D_a=a+U-2aU=(1-a)U+a(1-U)\sim \dU(a, 1-a),$$
where $U\sim \dU(0, 1)$. 
\end{proof}

Figure \ref{fig:degree_sequence} shows the degree proportion of the labelled vertices in the permutation graph associated with a random permutation of length $n=10^5$.  The symmetry in the figure around the {\it mid-vertex}, that is the vertex labeled $n/2$, is because the distribution of $d_n(i)$ and $d_n(n+1-i)$ are the same for every $i\in [n]$. This is confirmed by the above Corollary, which shows that for $r\in [0, 1]$, the degree proportion of the vertex labelled $\ceil{nr}$ converges to the uniform distribution over the interval $[r, 1-r]$. The shrinking length of this interval as $r$ approaches $1/2$, explains the fan-like structure of the degree proportion in Figure \ref{fig:degree_sequence}.  Moreover, when $r=1/2$ this implies that $d_n(\ceil{n/2})/n$ converges in probability to 1/2, as can be seen from Figure \ref{fig:degree_sequence}.

\begin{figure*}[h]
\centering
\begin{minipage}[c]{0.5\textwidth}
\centering
\includegraphics[width=3.0in]
    {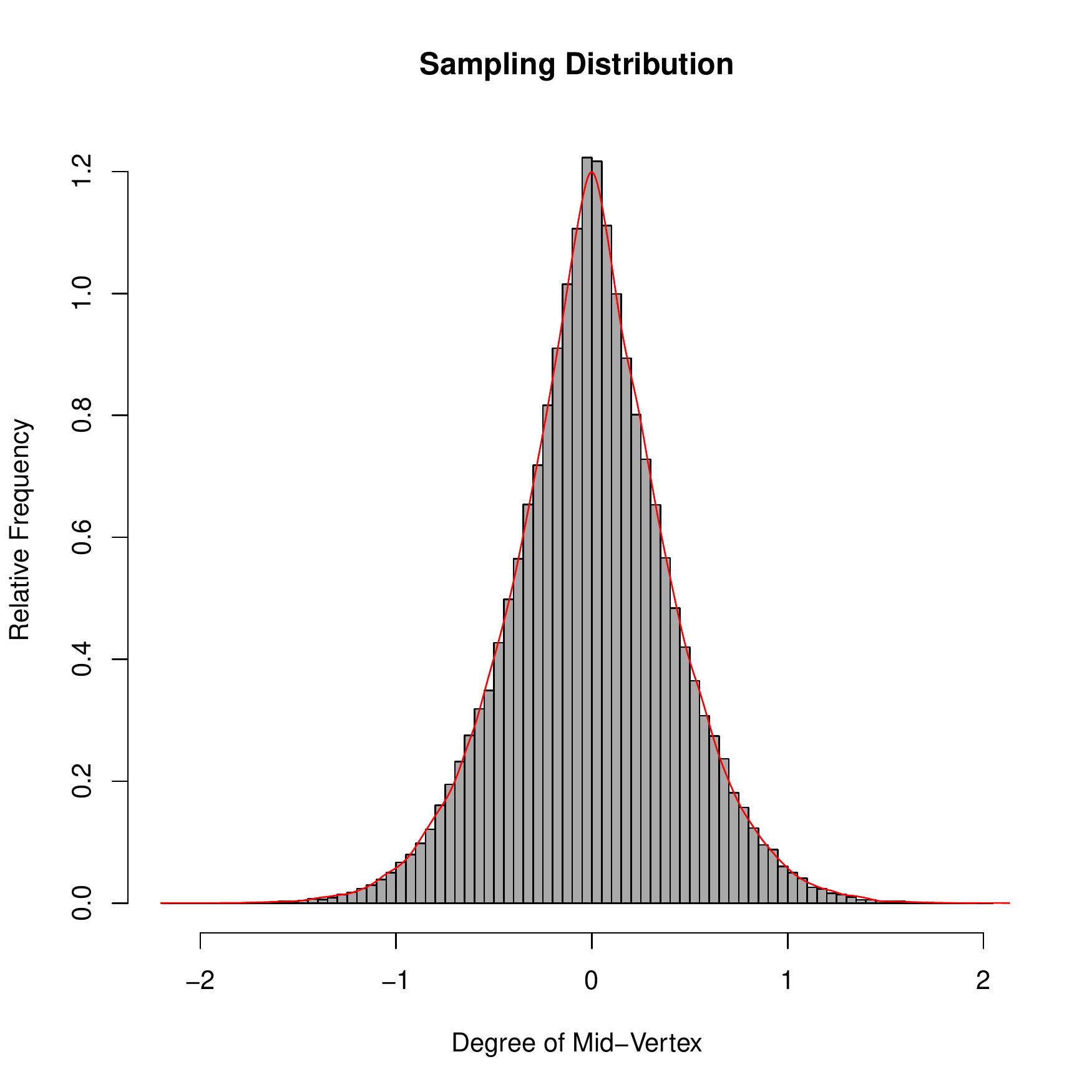}\\
\small{(a)}    
\end{minipage}
\begin{minipage}[c]{0.49\textwidth}
\centering
\includegraphics[width=3.0in]
    {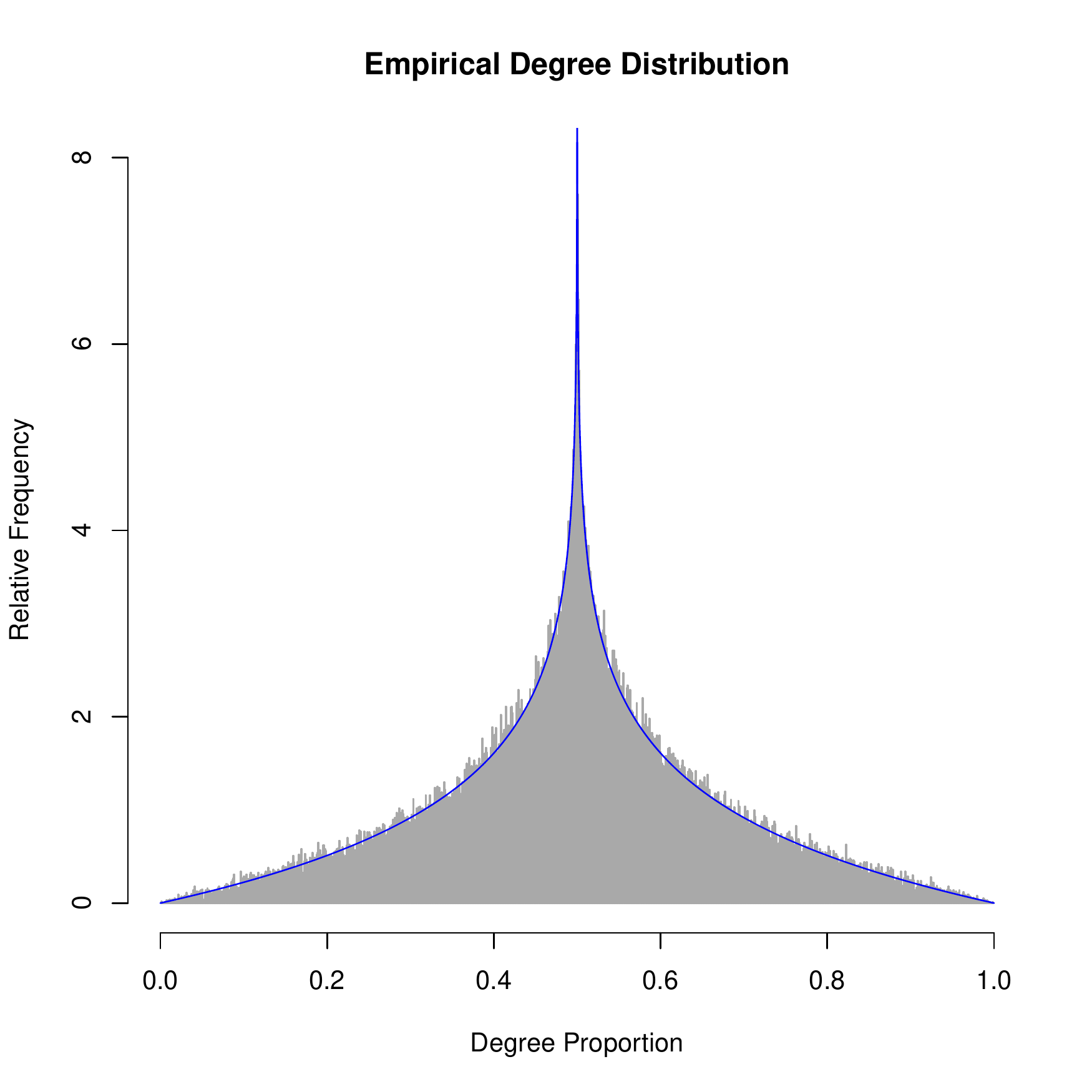}\\
\small{(b)}   
\end{minipage}
\caption{\small{(a) The sampling distribution of the degree of the mid-vertex (the vertex labeled $n/2$), for the permutation graph associated with a random permutation of length $n=10^4$, repeated over $10^5$ samples. (b) The empirical degree distribution of the permutation graph associated with a random permutation of length $n=10^5$. The blue curve represents the density of the limiting distribution $f(z)=-\log|2z-1|$. 
}}
\label{fig:degree_uniform}
\end{figure*}

It follows from Corollary \ref{cor:degreejoint} that the degree proportion of the mid-vertex, that is, the vertex labelled $\ceil{n/2}$, converges to 1/2 in probability. Therefore, it is reasonable to expect a central limit theorem for $d_n(\ceil{n/2})/n$ around 1/2 after an appropriate rescaling. This is detailed in the following theorem and illustrated in Figure \ref{fig:degree_uniform}(a). The proof is given in Section \ref{sec:midvertex}.

\begin{thm}Let $\pi_n\in S_n$ be a uniform random permutation and $d_n(1), d_n(2), \ldots d_n(n)$ be the degree sequence of the associated permutation graph $G_{\pi_n}$. Then 
\begin{equation}
\sqrt{n}\left(\frac{d_n(\ceil{n/2})}{n}-\frac{1}{2}\right)\stackrel{\sD}{\rightarrow}N(0, U(1-U)),
\label{eq:degreenormal}
\end{equation}
where $U\sim \dU[0,1]$.
\label{th:degreenormal}
\end{thm}

The empirical degree distribution of a uniformly random permutation graph is a direct consequence of the Proposition \ref{ppn:degree_nu}. The limiting density of the empirical degree distribution is depicted in Figure \ref{fig:degree_uniform}(b). The density is supported on $[0, 1]$ and has an interesting shape: it vanishes at the end points, and blows up to infinity at $z = 1/2$. 

\begin{cor}
Let $\pi_n\in S_n$ be a uniformly random permutation and $G_{\pi_n}$ the associated permutation graph. Then the empirical degree distribution 
\begin{equation}
\kappa(G_{\pi_n}):=\frac{1}{n}\sum_{i=1}^n\delta_{\frac{d_n(i)}{n}}\dto Z:=(1-U)V+U(1-V),
\label{eq:deglimit}
\end{equation}
 where $U, V$ are independent $\dU(0,1)$. Equivalently, $Z$ has the same distribution as $\dU(V, 1-V)$, where $V\sim \dU(0,1)$, and has a density with respect to Lebesgue measure given by $f_Z(z)=-\log|1-2z|$, for $0\le z\le 1$.
\label{cor:deglimit}
\end{cor}

The degree distribution of a uniformly random permutation graph has different variability depending on the label of the vertex. In fact, the next corollary shows that a uniformly random permutation graph is rather irregular: a typical vertex has degree around $n/2$, however the minimum and the maximum degrees are far apart:

\begin{cor}
For a uniformly random permutation $\pi_n$, $n^{-1}\delta(G_{\pi_n}) \stackrel{\sP}{\rightarrow}0$ and $n^{-1}\Delta(G_{\pi_n})\stackrel{\sP}\rightarrow 1$, where $\Delta(G_{\pi_n})$ and  $\delta(G_{\pi_n})$ are the maximum and the minimum degree of $G_{\pi_n}$, respectively.
\label{cor:max_min_degree}
\end{cor}

\begin{proof}
To prove the result for the minimum degree it suffices to show that for any $0<\varepsilon<1/2$, $\lim_{n\rightarrow\infty}\P(\delta(G_{\pi_n})> n\varepsilon)=0$. To this end, choose $r\ge 1$ arbitrary, and  real numbers $0< t_1<t_2<\cdots<t_r\le \varepsilon/4$. By Corollary \ref{cor:degreejoint} $$\lim_{n\rightarrow \infty}\P_n\left(\min_{1\le i\le r}d_n({\ceil {nt_i} })>n\varepsilon\right)=\prod_{i=1}^r\frac{1-\varepsilon-t_i}{1-2t_i}\le \left(\frac{1-\varepsilon}{1-\varepsilon/2}\right)^r$$
Since this is true for every $r$ and for all $\varepsilon>0$ fixed, letting $r\rightarrow\infty$ gives the result. The result for the maximum degree can be proved similarly. 
\end{proof}

In light of the above corollary, it is natural to ask whether $\delta(G_{\pi_n})$ converges to a non degenerate distribution after appropriate rescaling. The following theorem, which is proved in Section \ref{sec:min_deg} shows that the minimum degree of a uniformly random permutation graph is a Rayleigh distribution.

\begin{thm}\label{min_deg}
For a uniformly random permutation $\pi_n$,  
$$\frac{\delta(G_{\pi_n})}{\sqrt{n}}\stackrel{\sD}{\rightarrow}\Gamma,$$
where $\Gamma$ is the Rayleigh distribution with parameter $\frac{1}{\sqrt 2}$, that is, $\P(\Gamma>\gamma)=e^{-\gamma^2}$ for all $\gamma>0$.
\end{thm}

\subsection{Mallows Random Permutation}
\label{mallows}

One of the most popular non-uniform model on permutations, which has applications in statistics \cite{clmallows},  is the Mallows measure. For $\beta \in \R$, denote by $\pi_n\sim M_{\beta, n}$ the Mallows random permutation over $S_n$ with probability mass function
$$m_{\beta, n}(\sigma):=\frac{e^{-\beta\cdot\frac{\lambda(\sigma)}{n}}}{\sum_{\sigma\in S_n}e^{-\beta\cdot\frac{\lambda(\sigma)}{n}}},$$  
where $\lambda(\sigma)=|\{(i, j): (i-j)(\pi_n(i)-\pi_n(j))<0\}|$ is the number of inversions of the permutation $\sigma$. The uniform random permutation corresponds to the case $\beta=0$. It is easy to check that for $\pi_n\in S_n$ chosen uniformly at random, $\nu_{\pi_n}$ converges weakly in probability to $\dU(0, 1)$. This was generalized to Mallow random permutations by Starr \cite{starr}. Using this and Theorem \ref{thm:degreejoint} we compute the limiting density of the degree proportion in a Mallows random permutation. The limiting density exhibits interesting phase transitions depending on the value of $\beta$. This is summarized in the following theorem and proved later in Section \ref{sec:mallows}.

\begin{figure*}[h]
\centering
\begin{minipage}[c]{0.49\textwidth}
\centering
\includegraphics[width=3.2in]
    {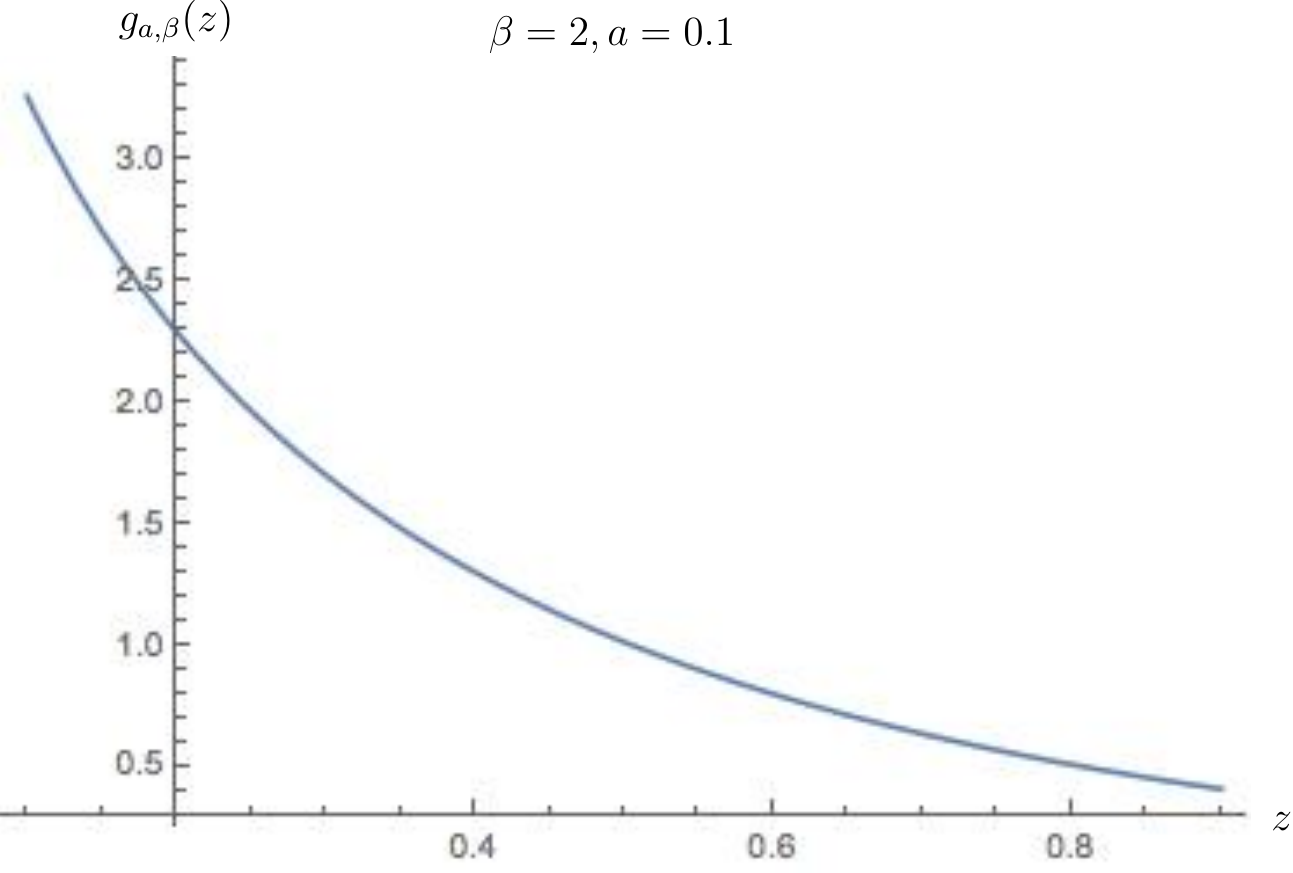}\\
\small{(a)}    
\end{minipage}
\begin{minipage}[c]{0.5\textwidth}
\centering
\includegraphics[width=3.5in]
    {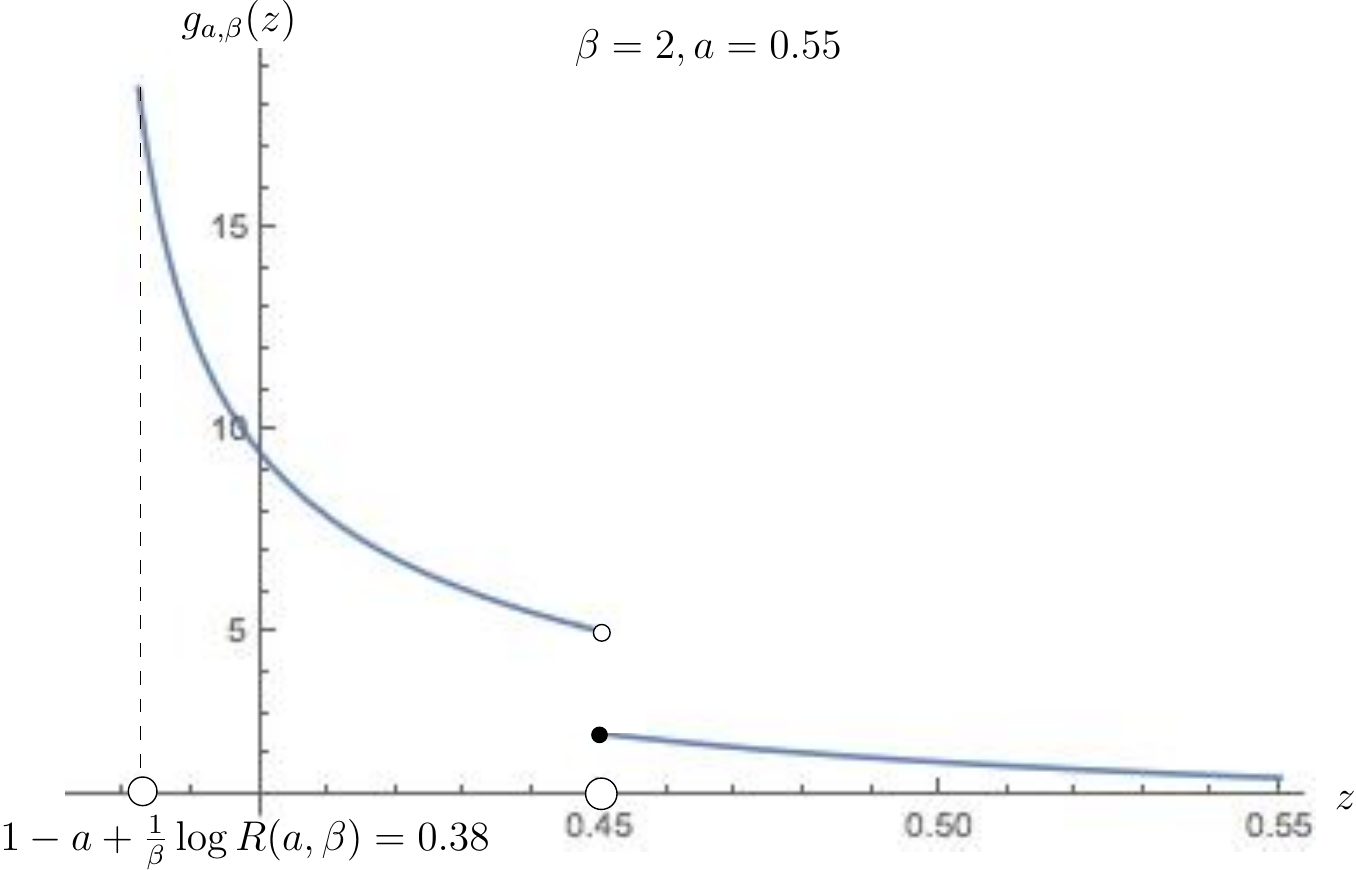}\\
\small{(b)}    
\end{minipage}
\caption{\small{Density of $D_{a, \beta}$ for $\beta=2$ and (a) $a=0.1$, and (b) $a=0.55$. For $\beta=2$, $a_c(\beta)=0.28311$. Since $0.1\notin (0.28311, 0.71689)$, the density (a) of $D_{0.1, 2}$ is a continuous function supported on $[0.1, 0.9]$. On the other hand, $0.55\in (0.28311, 0.71689)$, and  the density (b) of $D_{0.55, 2}$ is a piecewise continuous function with the pieces supported on $[0.38, 0.45)$ and $[0.45, 0.55]$, with a discontinuity at $1-a=0.45$. 
}}
\label{fig:density}
\end{figure*}

\begin{thm}Let $\beta \in \R$ and $\pi_n\sim M_{\beta, n}$. If $d_n(1), d_n(2), \ldots d_n(n)$ is the degree sequence of the associated permutation graph $G_{\pi_n}$, then the degree process 
$$d_n(\cdot)\stackrel{w}\Rightarrow Z(\cdot).$$
More precisely, for indices $0\leq  r_1<r_2< \cdots <r_s\leq 1$,
$$\left(\frac{d_n(\ceil{nr_1})}{n}, \frac{d_n(\ceil{nr_2})}{n}, \ldots, \frac{d_n(\ceil{nr_s})}{n}\right) \stackrel{\sD}{\rightarrow}(D_{1, \beta}, D_{2, \beta}, \ldots, D_{s, \beta}),$$
where $D_{1, \beta}, D_{2, \beta}, \ldots, D_{s, \beta}$ are independent, and
\begin{description}
\item[1] if $a\notin [a_c(\beta), 1-a_c(\beta)]$, $D_{a, \beta}$ has density: 
\begin{equation}
g_{a, \beta}(z)=\frac{\beta  e^{\frac{1}{2} \beta  (a-z)}}{(1-e^{-\beta})\left(e^{\beta  (a+z)}-e^\beta R(a, \beta)
\right)^{\frac{1}{2}}},  \quad z\in [a, 1-a];
\label{eq:density_2}
\end{equation}

\item[2] if $a\in (a_c(\beta), 1-a_c(\beta))$, $D_{a, \beta}$ has density: 
\begin{equation}
g_{a, \beta}(z)=
\left\{
\begin{array}{ccc}
\frac{\beta  e^{\frac{1}{2} \beta  (a-z)}}{(1-e^{-\beta})\left(e^{\beta  (a+z)}-e^\beta R(a, \beta)\right)^{\frac{1}{2}}}  & \text{ for } z\in \left[a, 1-a \right],    \\
\frac{2\beta  e^{\frac{1}{2} \beta  (a-z)}}{(1-e^{-\beta})\left(e^{\beta  (a+z)}-e^\beta R(a, \beta)\right)^{\frac{1}{2}}}   &  \text{ for } z\in \left[1-a+\frac{1}{\beta}\log R(a, \beta), a\wedge 1-a\right);
\end{array}
\right.
\label{eq:density_1}
\end{equation}
\end{description}
where 
\begin{equation}
R(a, \beta)=\frac{4 \left(e^{\beta }-e^{a \beta }\right) \left(e^{a \beta }-1\right)}{\left(e^{\beta }-1\right)^2}, \text{ and } a_c(\beta)=\frac{1}{2}-\frac{\log\cosh(\beta/2)}{\beta}.
\label{eq:Rab}
\end{equation}
\label{th:mallowsdegreejoint}
\end{thm}

\begin{figure*}[h]
\centering
\begin{minipage}[c]{1.0\textwidth}
\centering
\includegraphics[width=5.0in]
    {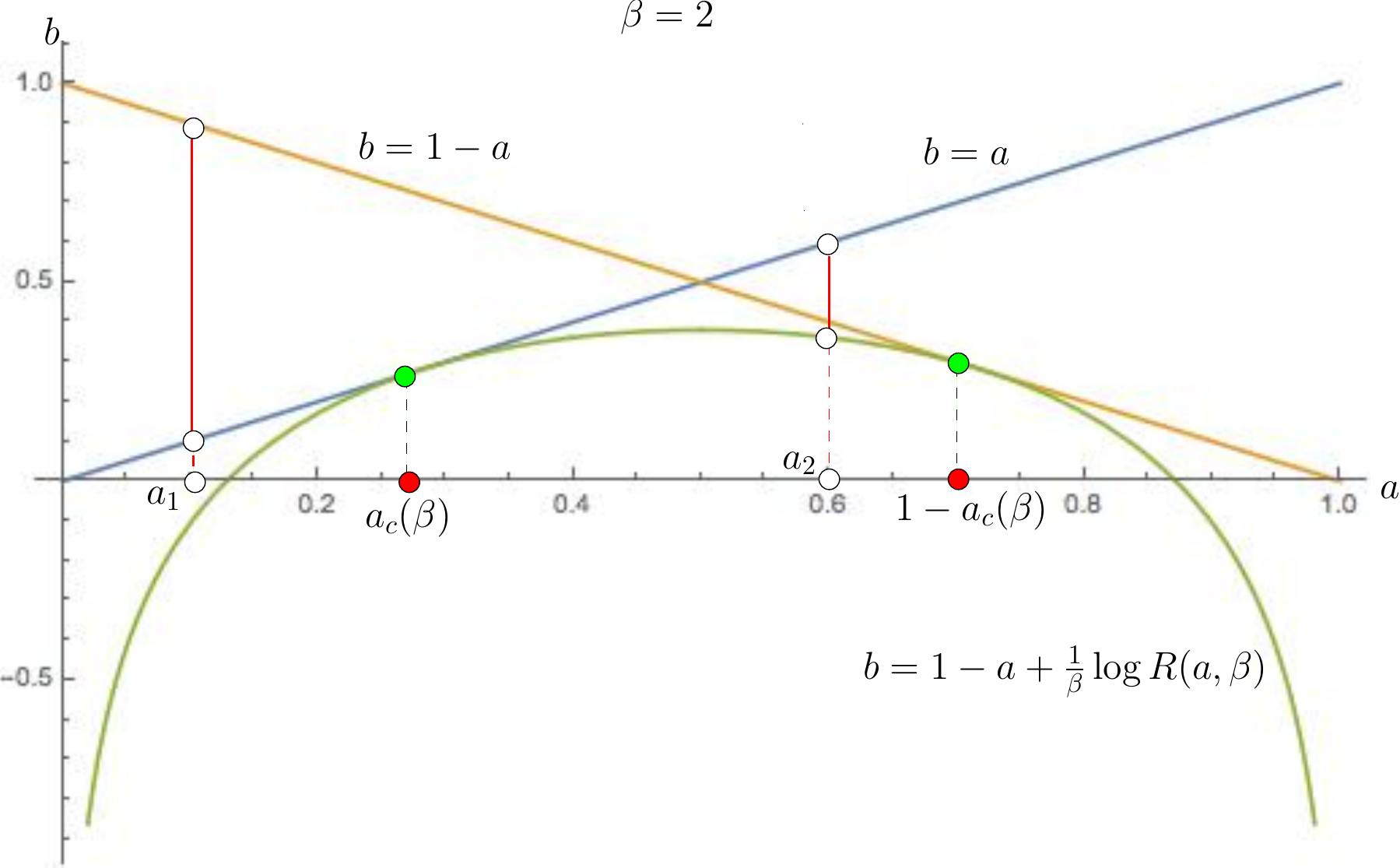}\\
\end{minipage}
\caption{\small{The support of $D_{a, \beta}$, for $\beta=2$ and  $a\in [0, 1]$:  The points of tangency of the curve  $b=1-a+\frac{1}{\beta}\log R(a, \beta)$ with the two straight lines $b=a$ and $b=1-a$ are colored green. The corresponding red points on the $x$-axis are the critical points $a_c(\beta)$ and $1-a_c(\beta)$.  For a fixed $a\in [0, 1]$, the support of $D_{a, \beta}$ is the interval intercepted  by the vertical line at $a$ either between the two straight lines (if $a\notin (a_c(\beta), 1-a_c(\beta))$, for example, when $a=a_1$), or between the curve and one of the two straight lines (if $a\in (a_c(\beta), 1-a_c(\beta))$, for example, when $a=a_2$).}}
\label{fig:range}
\end{figure*}

The above theorem gives the limiting distribution of the degree process of the permutation graph $G_{\pi_n}$ associated with a Mallows random permutation $\pi_n\sim M_{\beta, n}$. For $\beta\in \R$ fixed, the limiting distribution of $d_n(\ceil{na})/n$ has a phase transition depending on the value of $a\in [0, 1]$. There exists two {\it critical points} $a_c(\beta)$ and $1-a_c(\beta)$, such  that
for $a\notin [a_c(\beta), 1-a_c(\beta]$, the limiting density of $D_{a, \beta}$ is a continuous function supported on $[a, 1-a]$. However, for $a$ in the  {\it critical interval} $(a_c(\beta), 1-a_c(\beta))$, the density of $D_{a, \beta}$ breaks into two piecewise continuous parts on the intervals $$\left[1-a+\frac{1}{\beta}\log R(a, \beta), a\wedge 1-a\right), \text{ and } (a\wedge 1-a,  a\vee 1-a],$$ with a discontinuity at the point $a\wedge 1-a$.  Plots of the limiting density of $D_{a, \beta}$ are shown in Figure \ref{fig:density} for $\beta=2$ and $a=0.1$ and $a=0.55$. The changes in the support of $D_{a, \beta}$ for values of $a$ in the critical interval is depicted in Figure \ref{fig:range}.

The {\it critical curves} $\beta \mapsto a_c(\beta)$ and $\beta \mapsto 1-a_c(\beta)$ are shown in Figure \ref{fig:curve}. For a fixed $\beta_0\in \R$ the critical interval $(a_c(\beta_0), 1-a_c(\beta_0))$ is the interval between the 2 curves  intercepted by the vertical line at $\beta_0$. Note that for $\beta=0$, the $a_c(\beta)=1-a_c(\beta)=1/2$, that is, the  critical interval is empty. Therefore, for a uniform random permutation, the limiting density has no phase transition, as elaborated in Corollary \ref{cor:degreejoint}.

\begin{figure*}[h]
\centering
\begin{minipage}[c]{1.0\textwidth}
\centering
\includegraphics[width=4.0in]
    {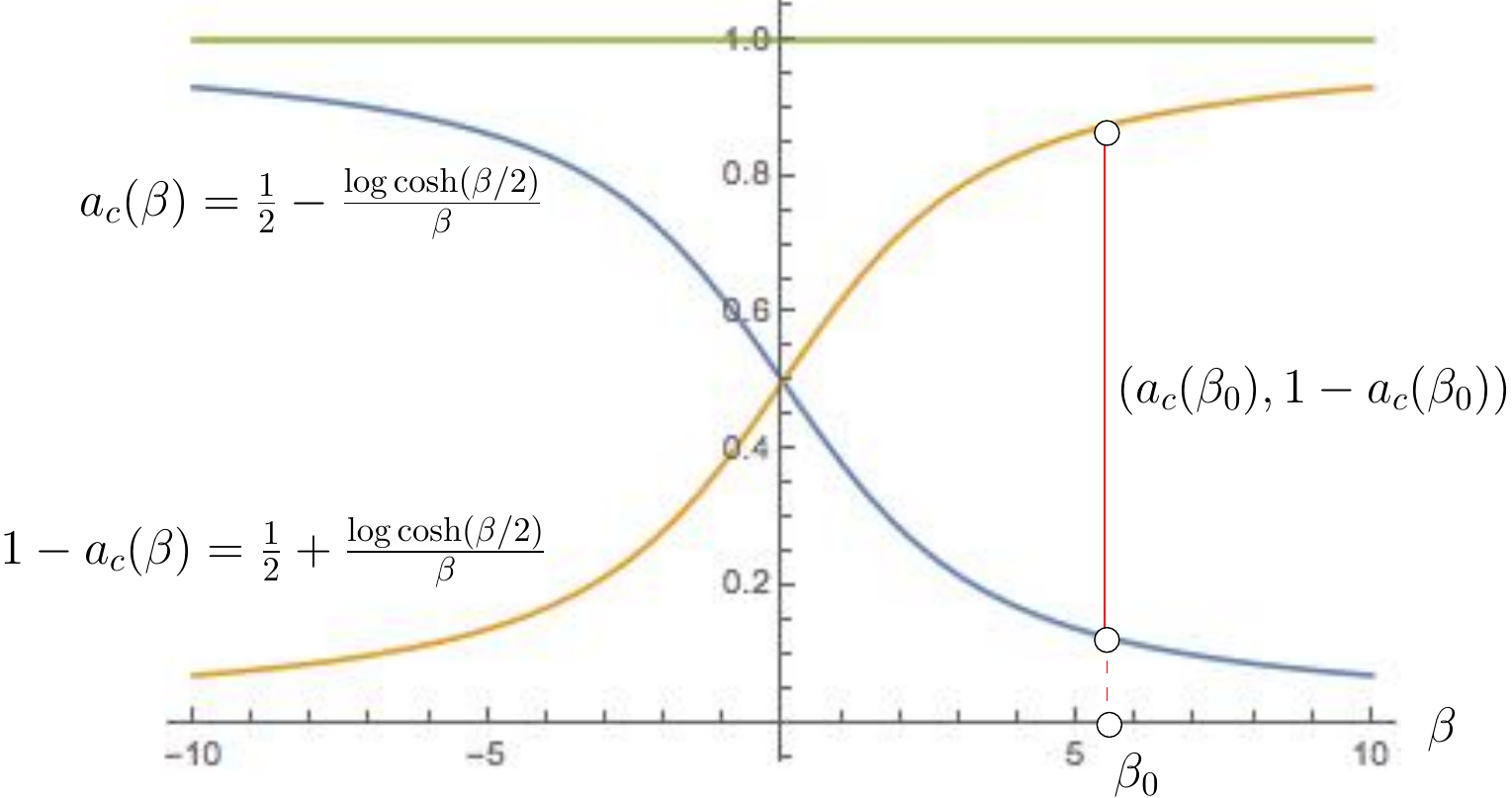}\\
\end{minipage}
\caption{\small{The critical transition curves $a_c(\beta)$ and $1-a_c(\beta)$ for $\beta\in [-10, 10]$. For a fixed $\beta_0\in \R$ the critical interval $(a_c(\beta_0), 1-a_c(\beta_0))$ is the interval between the 2 curves  intercepted by the vertical line at $\beta_0$.}}
\label{fig:curve}
\end{figure*}

The phase transition in the density of $D_{a, \beta}$ can be reinterpreted by fixing $a\in [0, 1]$ and varying $\beta$: Theorem \ref{th:mallowsdegreejoint} shows that for $a\in [0, 1]$ fixed, there exists a critical point $\beta_c(a)$ (obtained by solving for $\beta$ in $a_c(\beta)=a$) such that for $\beta \in[0, \beta_c(a)]$, the density of $D_{a, \beta}$ is a continuous function supported on $[a, 1-a]$. However, for $\beta > \beta_c(a)$ the density of $D_{a, \beta}$ breaks into two piecewise continuous parts with a discontinuity at the point $a\wedge 1-a$. If $\beta=1/T$ denotes the inverse temperature, then this phenomenon is the effect of {\it replica symmetry breaking} in statistical physics as one moves from the high temperature to the low temperature regime.

\section{Degree Distribution of Random Permutations}
\label{sec:main}

In this section we prove the convergence of the degree process whenever the permutation process converges.
Recall that the Kolmogorov-Smirnov distance  on the space of probability measures on $[0,1]^2$ is defined by
$$||\nu-\mu||_{\mathrm KS}:=\sup_{0\le x,y\le 1}|F_\nu(x,y)-F_\mu(x,y)|,$$
where $F_\nu$ and $F_\mu$ are the bivariate distribution functions of $\nu$ and $\mu$ respectively. Note that convergence in Kolmogorov-Smirnov distance implies weak convergence, but not conversely.  Finally, for $\pi_n\in S_n$ define the {\it empirical permutation measure} 
\begin{equation}
\tilde \nu_{\pi_n}=\frac{1}{n}\sum_{i\in [n]}\delta_{\left(\frac{i}{n}, \frac{\pi_n(i)}{n}\right)}.
\label{mu:emp}
\end{equation}
It is easy to check that the permuton $\nu_{\pi_n}$ associated with the permutation $\pi_n$ satisfies $||\nu_{\pi_n}-\tilde \nu_{\pi_n}||_{\mathrm{KS}}\pto 0$, and so
a sequence of permutations $\pi_n\in S_n$ converges to a permuton $\nu$ if and only if $\tilde \nu_{\pi_{n}}$ converges weakly in probability to $\nu$.

The following theorem shows that the convergence of the permutation process implies the convergence of the permutons. This connection between the two types of convergence might give new insights into the permutation limit theory described in Section \ref{sec:graph_permutation_limit}.

\begin{thm}
Let $\pi_n\in S_n$ be a sequence of random permutations such that 
\begin{eqnarray}\pi_n(\cdot)\stackrel{w}{\Rightarrow}Z(\cdot).
\label{eq:convprocess}
\end{eqnarray}
Then $\left(\tilde \nu_{\pi_n},\pi_n(\cdot)\right)$ converges jointly weakly in distribution. In particular, there exists a (possibly random) measure $\mu\in\cM$, such that the permuton sequence $(\nu_{\pi_n})_{n\geq 1}$ converges in distribution to $\mu\in \cM$.	
\label{thm:cond}
\end{thm}

\begin{proof}Fix a continuous function $f: [0,1]^2\rightarrow [0, 1]$,  positive integers $a, b\ge 1$,  continuous functions $g_1, g_2, \cdots,g_b:[0,1]\mapsto \R$, and real numbers $s_1, s_2,\cdots, s_b \in (0,1]$.
The joint $a$-th moment of $\nu_{\pi_n}(f)$ and  $(g_1(\pi_n(s_1)), g_2(\pi_n(s_2)), \cdots, g_n(\pi_n(s_b)))$ is 
\begin{eqnarray}
\E\left(\tilde\nu_{\pi_n}^a(f)\prod_{j=1}^b g_j(\pi_n(s_j))\right)&=&\frac{1}{n^a}\E\left(
\sum_{i_1, i_2, \ldots, i_a\in [n]}\prod_{j=1}^af\left(\frac{i_j}{n}, \frac{\pi_n(i_j)}{n}\right)\prod_{j=1}^b g_a(\pi_n(s_j))\right)\nonumber\\
&=&\E\left(\prod_{j=1}^b g_j(\pi_n(s_j)) \int_{[0,1]^a}\prod_{j=1}^a f(x_j, \pi_n(x_j)) dx_j\right)+o(1),
\label{eq:proI}
\end{eqnarray}
where the last equality follows from the uniform continuity of $f$.

Now, by assumption (\ref{eq:convprocess}),
$$\prod_{j=1}^b g_j(\pi_n(s_j)) \prod_{j=1}^a f(x_j, \pi_n(x_j))\stackrel{\sD}{\rightarrow} \prod_{j=1}^b g_j(Z(s_j)) \prod_{j=1}^a f(x_j, Z(x_j)).$$
Therefore, by the Dominated Convergence Theorem 
\begin{align*}
\lim_{n\rightarrow \infty}\E\left(\prod_{j=1}^b g_j(\pi_n(s_j)) \prod_{j=1}^a f(x_j, \pi_n(x_j))\right)=\E\left(\prod_{j=1}^b g_j(Z(s_j)) \prod_{j=1}^a f(x_j, Z(x_j))\right).	
\end{align*}
The RHS above is measurable in $(x_1, x_2, \cdots, x_a)$, as it is the limit of measurable functions. Another  application of Dominated Convergence Theorem gives	
\begin{eqnarray}
\E\left(\tilde\nu_{\pi_n}^a(f)\prod_{j=1}^b g_j(\pi_n(s_j))\right)&=&\int_{[0,1]^a}\E\left(\prod_{j=1}^b g_j(\pi_n(s_j)) \prod_{j=1}^ a f(x_j, \pi_n(x_j))\right)\prod_{j=1}^a dx_j+o(1),\nonumber\\
&=&\int_{[0,1]^a}\E\left(\prod_{j=1}^b g_j(Z(s_j)) \prod_{j=1}^a f(x_j, Z(x_j))\right)\prod_{j=1}^a dx_j+o(1).
\label{eq:proII}
\end{eqnarray}
Since (\ref{eq:proII}) holds for all choices of $f$ and $g_1,\cdots, g_b$, the joint convergence in law of $(\tilde \nu_{\pi_n} ,\pi_n(\cdot))$ follows. Since the space $\cM$ is closed,  implies $\tilde \nu_{\pi_n}\dto \mu$, for some random measure $\mu \in \cM$.

Finally, since $||\nu_{\pi_n}-\tilde \nu_{\pi_n}||_{\mathrm{KS}}\rightarrow 0$ in probability, $(\nu_{\pi_n})_{n\geq 1}$ converges in distribution to the permuton $\mu \in \cM$.
\end{proof}

\subsection{Proofs of Theorem \ref{thm:degreejoint} and Corollary \ref{cor:independent}}
The above theorem can now be used to prove the convergence of the degree process.

\subsubsection{Proof of Theorem \ref{thm:degreejoint}}
\label{sec:thmain}
By Theorem \ref{thm:cond} we have $$(\tilde \nu_{\pi_n}, \pi_n(\cdot))\dto(\mu,Z(.)),$$
where $\mu\in \cM$ is a random measure such that $\tilde \nu_{\pi_n}\dto \mu$, and $Z(.)$ is a stochastic process on $(0,1]$ such that $\pi_n(\cdot) \dto Z(\cdot)$.  Thus, fixing $k\ge 1$ and $0<t_1<t_2<\cdots<t_k\le 1$ we have 
$$(\tilde\nu_{\pi_n},\pi_n(t_1),\cdots,\pi_n(t_k))\stackrel{\sD}{\rightarrow}(\mu,Z(t_1),\cdots,Z(t_k)).$$
Applying Skorohod's representation theorem on the separable metric space $\cM\times [0,1]^k$ (see Billingsley \cite[Theorem 6.7]{billingsley}), without loss of generality assume that the above convergence happens almost surely.

Now for any $t\in (0,1]$,
\begin{align}
a_n(t)=\frac{1}{n}\sum_{a=1}^{\lceil nt\rceil }\pmb 1\{\pi_n(a)>\pi_n(\lceil nt\rceil)\}=\tilde \nu_{\pi_{n}}\Big([0, \lceil nt\rceil /n]\times(\pi_n(t),1]\Big)=&\tilde\nu_{\pi_{n}}\Big([0, t]\times(\pi_n(t),1]\Big)+o(1)\nonumber\\
=&t-F_{{\nu}_{\pi_n}}(t,\pi_n(t))+o(1),
\label{eq:an}
\end{align}
where the last step uses $||\tilde{\nu}_{\pi_n}-\nu_{\pi_n}||_{\mathrm{KS}}=o(1)$.  By a similar argument,\begin{align}
b_n(t)=\pi_n(t)-F_{{\nu}_{\pi_n}}(t,\pi_n(t))+o(1). 
\label{eq:bn}
\end{align}
Combining (\ref{eq:an}) and (\ref{eq:bn}), for any $1\le i\le k$ we have
\begin{align*}
d_n(t_i)=a_n(t_i)+b_n(t_i)=&t_i+\pi_n(t_i)-2F_{{\nu}_{\pi_n}}(t_i,\pi_n(t_i))+o(1)\nonumber\\
=&t_i+\pi_n(t_i)-2F_\mu(t_i,\pi_n(t_i))+||\nu_{\pi_n}-\mu||_{\mathrm{KS}}+o(1)\\
\stackrel{a.s.}{\rightarrow}& t_i+Z(t_i)-2F_\mu(t_i, Z(t_i)), 
\end{align*}
where the last step uses $||\nu_{\pi_n}-\mu||_{\mathrm{KS}}\stackrel{a.s.}{\rightarrow}0$ (see Hoppen et al. \cite[Lemma 2.1]{permutation_limits}), and the fact that the function $F_\mu$ is continuous in each co-ordinate when the other co ordinate is held fixed. Indeed, this follows from the observation that any $\mu\in\cM$ has continuous marginals. 
Thus we have
$$(d_n(t_1),\cdots,d_n(t_k))\stackrel{a.s.}{\rightarrow}(t_1+Z(t_1)-2F_\mu(t_1,Z(t_1)),\cdots,t_k+Z(t_k)-2F_\mu(t_k,Z(t_k))),$$
from which finite dimensional convergence of $d_n(\cdot)$ follows.

\subsubsection{Proof of Corollary \ref{cor:independent}}
\label{sec:cormain}
%
%
From (\ref{eq:proI}) and (\ref{eq:proII}), it follows that for any continuous function $f:[0, 1]^2\rightarrow [0, 1]$ and $a\ge 1$,
$$\lim_{n\rightarrow\infty}\E\tilde{\nu}_{\pi_n}(f)^a=\int_{[0,1]^a}\E\left(\prod_{j=1}^a f(x_j, Z(x_j))\right)\prod_{j=1}^a dx_j.$$
Under the assumption of independence, the RHS above equals $\left(\int_0^1 \E f(x,Z(x))dx\right)^a$, 
which implies that $\tilde{\nu}_{\pi_n}(f)$ converges in probability to the non random quantity $$\int_0^1 \E f(x, Z(x))dx=\E_\mu f(X,Y),$$
where $(X,Y)\sim \mu$ is as required. Since this holds for all continuous functions $f$ the desired conclusion follows.

The independence of the finite dimensional marginals of $Z(\cdot)$, implies the same for the degree process $D(\cdot)$ by (\ref{eq:deg_conv}).

\subsection{A Dependent Degree Process}

Even though theorem \ref{thm:degreejoint} allows for $\mu$ to be random, in most examples in this paper $\mu$ turns out to be non-random and the corresponding degree process has independent finite dimensional distributions. In this section we construct a sequence of random permutations where the limiting permuton is random and the finite dimensional distributions of the degree process are not independent:

\begin{ex}
\label{dependent}
Suppose $W_n$ is a uniform random variable on $[n]:=\{1, 2, \ldots, n\}$, and $\pi_n\in S_n$ defined by \begin{equation}
\pi_n(i):=(i+W_n-1 \mod n)+1.
\label{eq:permcyclic}
\end{equation}
Note that $\pi_n$ is a cyclic shift of the identity permutation, where the length of the shift is chosen uniformly random. 
\end{ex}

\begin{ppn}Let $(\pi_n)_{n\geq 1}$ be a sequence of random permutations as defined in (\ref{eq:permcyclic}). Then the degree process 
\begin{align}
d_n(t)\stackrel{w}\Rightarrow D(t):=W\cdot \pmb 1\{W+t<1\}+(1-W)\cdot \pmb 1\{W+t\ge 1\},
\label{eq:degprocess}
\end{align}
where $W\sim \dU[0, 1]$.
\label{ppn:permcyclic}
\end{ppn}

\begin{proof}We will first show that the permutation process $\pi_n(\cdot)$ converges weakly in distribution. For $s \ge 1$ and let $g_1,g_2,\cdots, g_s$ be continuous functions on $[0,1]$. Then, 
$$\E g_1\Big(\pi_n(i_1)\Big)\cdots g_s\Big(\pi_n(i_s)\Big)\rightarrow \int_0^1 g_1(i_1+u\mod 1)\cdots g_s(i_s+u\mod 1)du.$$
Hence, $\pi_n(\cdot)\stackrel{w}{\Rightarrow}Z(\cdot)$, where $Z(\cdot)$ is a stochastic process defined by 
\begin{align}
Z(t)=W+t\mod 1, \quad \text{ with } \quad W\sim \dU[0,1]. 
\label{eq:permprocess}
\end{align}

By Theorem \ref{thm:degreejoint}  and by (\ref{eq:proI}) we have  $\nu_{\pi_n}\dto \mu$ such that for any continuous function $f$ on the unit square,
\begin{align*}
\E\mu(f)^a=\int_{[0,1]^a}\E \left( \prod_{j=1}^af(x_j, Z(x_j))\right) \prod_{j=1}^a dx_j=&\int_{[0,1]^{a+1}} \prod_{j=1}^af(x_j, \omega+x_j\mod 1)\prod_{j=1}^a dx_j d\omega\\
=&\int_{0}^{1} \left(\int_0^1 f(x,x+\omega \mod 1)dx\right)^a d\omega
\end{align*}
Therefore, the limiting measure $\mu$ is random and has the following law: For every $s\in [0,1]$, let $\kappa_s$ be the joint law of $(V, s+V\mod 1)$, where $V\sim \dU[0,1]$. Then $\mu=\kappa_W$ with  $W \sim [0,1]$.
		
To compute the limit of the degree process, we compute the distribution function of $\kappa_s$. In this case, with $U\sim \dU[0,1]$ and $0\le a,b\le 1$
\begin{align}
F_{\kappa_s}(a,b)=&\P(U\le a, U+s\mod 1\le b)\nonumber\\
=&\P(U\le a, U+s\le b)+\P(U\le a,1\le U+s\le b+1)\nonumber\\
=&\min(a,b-s)_++\min(a+s-1,b)_+,
\label{eq:FsigmaI}
\end{align}
which implies
\begin{align}
F_{\kappa_W}(t,Z(t))=&\min(t,(W+t\mod 1)-W)_++\min(t+W-1,W+t\mod 1)_+\nonumber\\
=&t\cdot \pmb 1\{W+t< 1\}+(t+W-1)\pmb 1\{W+t\ge 1\}.
\label{eq:FsigmaII}
\end{align}
Therefore, by Theorem (\ref{thm:degreejoint}) $d_n(\cdot)\stackrel{w}\Rightarrow D(\cdot)$ where: $W\sim U[0,1]$, and 
\begin{align}
D(t):=&t+(W+t\mod 1)-2F_{\kappa_W}(t, W+t\mod 1)\nonumber\\
=&W\cdot \pmb 1\{W+t<1\}+(1-W)\cdot \pmb 1\{W+t\ge 1\}.
\label{eq:degprocess}
\end{align}
In this case the finite dimensional distributions of the permutation process (\ref{eq:permprocess}) and the degree process (\ref{eq:degprocess}) are not independent, and the limiting permuton is random.
\end{proof}

\section{Uniformly Random Permutation Graph}
\label{sec:uniform}

Limiting properties of permutation statistics associated with a uniformly random permutation are widely studied. The number of edges $|E(G_{\pi_n})|$ is the number of inversions in the  permutation $\pi_n$. For a uniformly random permutation the distribution of the number of inversions is $\sum_{i=1}^{n}X_i$, where the random variables $X_i$ are independent and uniformly distributed over the set $\{0, 1,\ldots, i-1\}$, for every $i \in \{1, 2, \ldots, n\}$. Normal approximations to $|E(G_{\pi_n})|$ can be proved using these results and standard versions of the central limit theorem (refer to Fulman \cite{fulman_stein} for a proof using the method of exchangeable pairs). The largest clique in a permutation graph corresponds to the longest decreasing subsequence in the permutation. Similarly, an increasing subsequence in a permutation corresponds to an independent set of the same size in the corresponding permutation graph. Asymptotics for the maximum clique and the independent set in $G_{\pi_n}$ follow from the seminal work of Baik et al. \cite{baik} on the length of the longest increasing subsequence in a uniformly random permutation.

The convergence of the degree process of a uniformly random permutation graph (Corollary \ref{cor:degreejoint}) is a direct consequence of the general theorem. The empirical degree distribution can also be easily derived. 

\subsection{Proof of Corollary \ref{cor:deglimit}} When $\pi_n$ is a uniformly random permutation, the limiting permuton is $\nu= \dU (0, 1) \times \dU (0, 1)$ and the result follows from Proposition \ref{ppn:degree_nu} by direct substitution. 

To get the density of the limiting random variable, let $Z:=(1-U)V+U(1-V)$ where $U, V$ are independent $\dU[0,1]$. For $z\leq 1/2$, conditioning on $U$ the distribution function of $Z$ can be calculated as
$$\P(Z\leq z)=\int_{0}^z \frac{z-u}{1-2u}du+\int_{1-z}^1 \frac{z-(1-u)}{2u-1}du.$$
Simplifying and differentiating the above expression with respect to $z$ gives the desired density for $z\leq 1/2$. For $z> 1/2$, the density can be derived similarly. The density vanishes at the end points, and blows up to infinity at $z=1/2$.

%

\subsection{A Hypergeometric Estimate} For studying the degree sequence of a uniformly random permutation graph, properties of the random variables $a_n(i)$ and $b_n(i)$ (defined in Section \ref{sec:summary}) will be needed. To this end, recall the hypergeometric distribution: A non negative integer valued random variable $X$ is said to follow the {\it hypergeometric distribution} with parameters $(N,M,r)$ if  
$$\P(X=x)=\frac{{M\choose x}{N-M\choose r-x}}{{N\choose r}}, \text{ for } x\in [\max\{0, r+M-N\}, \min\{M, r\}]$$
where $N\geq \max \{M, r\}$.

Note that $(i-1)-a_n(i)+b_n(i)=\pi_n(i)-1$, thus giving the simple relation $b_n(i)-a_n(i)=\pi_n(i)-i$.  Using this relation, the following proposition gives a concentration result for $d_n(i)$ around its conditional mean given $\pi_n(i)$.

\begin{ppn}\label{hyper}
The conditional distribution of $a_n(i)|\{\pi_n(i)=j\}$ is hypergeometric with parameters $(n-1,i-1,n-j)$. Consequently, for $R>0$
$$\P\left(\left|d_n({i})-\frac{(i-1)(n-\pi_n(i))-(\pi_n(i)-1)(n-i)}{n-1}\right|>R\Big|\pi_n(i)=j\right)\le 2e^{-\frac{R^2}{2n}}.$$
\end{ppn}

\begin{proof}
Given $\pi_n(i)=j$ and $a_n(i)=a$, $b_n(i)=a+j-i=:b$, and so
$$\P(a_n(i)=a|\pi_n(i)=j)=\frac{(i-1)!(j-1)!(n-i)!(n-j)!}{a!(i-1-a)!b!(n-i-b)!(n-1)!}=\frac{{i-1\choose a}{(n-1)-(i-1)\choose (n-j)-a}}{{n-1\choose n-j}},$$
and so $a_n(i)$ follows the hypergeometric distribution with aforementioned parameters.

Therefore, $\E(d_n(i)|\pi_n(i)=j)=\E(a_n(i))+\E(b_n(i))=\frac{(i-1)(n-j)+(j-1)(n-i)}{n-1}$. To prove the second conclusion note that 
$$\left|d_n(i)-\frac{(i-1)(n-j)+(j-1)(n-i)}{n-1}\right|> R\Leftrightarrow \Big|a_n(i)-\frac{(i-1)(n-j)}{n-1}\Big|>\frac{R}{2}.$$
An application of the bound in \cite{hyper_conc} now gives the desired conclusion.
\end{proof}

\subsection{CLT for the Mid-Vertex: Proof of Theorem \ref{th:degreenormal}} 
\label{sec:midvertex}
Let $Z_n=\sqrt{n}\left(\frac{d_n(\ceil{n/2})}{n}-\frac{1}{2}\right)$. Now, fixing $\delta>0$
\begin{eqnarray}
\P(Z_n\le x)&=&\frac{1}{n}\sum_{j=1}^n\P(Z_n\le x|\pi_n(\ceil{n/2})=j)\nonumber\\
&=&\frac{1}{n}\sum_{n\delta\le j\le n(1-\delta)}\P(Z_n\le x|\pi_n(\ceil{n/2})=j)+O(\delta)\nonumber\\
&=&\frac{1}{n}\sum_{n\delta \le j\le n(1-\delta)}\P\left(\frac{a_n(\ceil{n/2})-\frac{(\ceil{n/2}-1)(n-j)}{n-1}}{n^{\frac{1}{2}}}\le \lambda_n(x, j)\Big|\pi_n(\ceil{n/2})=j\right)+O(\delta),\nonumber\\
\label{hyper1}
\end{eqnarray}
where $\lambda_n(x, j)$ satisfies $\lim_{n\rightarrow \infty}\max_{n\delta \le j\le n(1-\delta)}|\lambda_n(x, j)-x/2|=0,$ for all $x \in \R$.

By Proposition \ref{hyper},
\begin{eqnarray}
\sigma_n^2(\ceil{n/2}),j)&:=&\Var(a_n(\ceil{n/2})|\pi_n(\ceil{n/2}))=j)\nonumber\\
&=&\frac{(\ceil{n/2})-1)(j-1)(n-\ceil{n/2}))(n-j)}{(n-1)^2(n-2)}\ge C(\delta)n,
\end{eqnarray}
for some $C(\delta)>0$, and for all $j\in [n\delta, n(1-\delta)]$. Using the Berry-Esseen theorem for hypergeometric distribution \cite[Theorem 2.2]{lahiri_chatterjee}, there exists a universal constant $C$ such that with  $C'(\delta):=C/\sqrt{C(\delta)}<\infty$,
\begin{align}\label{hyper2}
\left|\P\left(\frac{a_n(\ceil{n/2})-\frac{(\ceil{n/2}-1)(n-j)}{n-1}}{n^{\frac{1}{2}}}\le \lambda_n(x, j) \Big|\pi_n(i)=j\right)-\Phi\left(\sqrt{n}\cdot \frac{\lambda_n(x, j)}{\sigma_n(\ceil{n/2},j)}\right)\right| \le \frac{C'(\delta)}{n^{\frac{1}{2}}}.
\end{align}
Finally, note that $$\max_{n\delta\le j\le n(1-\delta)}\left|\frac{\sigma_n^2(\ceil{n/2},j)}{n}-\frac{j(n-j)}{4n^2}\right|=o(1),$$
where the $o(1)$ term goes to zero as $n\rightarrow \infty$. Moreover, since the function $\Phi$ is uniformly continuous on $\R$,
\begin{align}\label{hyper3}
\max_{n\delta \le j\le n(1-\delta)}\left|\Phi\left(\sqrt{n}\cdot \frac{\lambda_n(x, j)}{\sigma_n(\ceil{n/2},j)}\right)-\Phi\left(\frac{x}{\sqrt{(j/n)(1-j/n)}}\right)\right|=o(1).
\end{align}

Combining (\ref{hyper1}), (\ref{hyper2}) and (\ref{hyper3}) we have
$$\P(Z_n\le x)=\frac{1}{n}\sum_{n\delta \le j\le n(1-\delta)}\Phi\left(\frac{x}{\sqrt{(j/n)(1-j/n)}}\right)+o(1)+O(\delta).$$
On taking limits as $n\rightarrow\infty$  followed by $\delta\rightarrow 0$ we have
$$\lim_{n\rightarrow\infty}\P(Z_n\le x)=\int_{0}^1\Phi\left(\frac{x}{\sqrt{u(1-u)}}\right)du,$$
which completes the proof of the theorem.

\section{Convergence of the Permutation Process}
\label{sec:equic}

The convergence of the permutation process requires some regularity assumptions. In this section we introduce the notion of equicontinuity for a sequence of random permutations, and verify this for most standard exponential models on permutations.

\subsection{Equicontinuous Permutations}  We begin by recalling few definitions: 

\begin{defn} Let  $\sF$ be a family of functions from $[0,1]$ to $\R$. The family $\sF$ is {\it equicontinuous at a point} $x_0 \in [0,1]$ if for every $\varepsilon > 0$, there exists a  $\delta > 0$ such that $$|f(x_0)- f(x)| < \varepsilon \text { for all } f\in \sF \text { and all } x  \text{ such that }|x-x_0| < \delta.$$ The family is {\it equicontinuous} if it is equicontinuous at each point of $[0,1]$.
\end{defn}

Denote by $\cC[0, 1]$ the set of all continuous functions from $[0, 1]$ to $\R$. Similar to the notion of equicontinuity of a class of functions, we introduce the following notion of equicontinuity of permutations. 

\begin{defn}
Let $a\geq 1$ and $g_1, g_2, \ldots, g_a\in \cC[0, 1]$. For any sequence $\pi_n\in S_n$ of random permutations, define the function $G_{g_1, g_2, \ldots, g_a}: [0, 1]^a\rightarrow \R$
\begin{equation}
G^{(n)}_{g_1, g_2, \ldots, g_a}(s_1, s_2,  \ldots, s_a)=\E\left(\prod_{i=1}^a g_i(\pi_n(s_i))\right).
\label{eq:G}
\end{equation}
A sequence $(\pi_n)_{n \geq 1}$ of random permutations is said to be {\it equicontinuous} if the family $\{G^{(n)}_{g_1, g_2, \ldots, g_a}\}_{n\geq 1}$ is equicontinuous for all $g_1, g_2, \ldots, g_a\in \cC[0, 1]$ and $a\ge 1$.
\end{defn}

\begin{ppn}\label{new_results}
Let $\pi_n\in S_n$ be a sequence of random equicontinuous permutations such that $\nu_{\pi_n}\dto \mu$. Then the permutation process 
$$\pi_n(\cdot)  \stackrel{w}  \Rightarrow Z(\cdot),$$ where the finite dimensional distribution of $Z(\cdot)$ is as follows: Let $(X_1, Y_1), (X_2, Y_2), \cdots (X_a, Y_a)$ be independent draws from the random measure  $\mu\in \cM$. Then
\begin{equation}
\sL(Z(s_1), Z(s_2), \cdots, Z(s_a))\sim \sL( Y_1 |X_1=s_1, Y_2|X_2=s_2, \cdots, Y_a |X_a=s_a),
\label{eq:Z}
\end{equation}
for $0<s_1<s_2 \cdots<s_a\le 1$.
\end{ppn}

\begin{proof}Fix $a\geq 1$. For every $n\geq 1$ define a collection of random variables $\{(U_{j ,n}, V_{j,n})\}^a_{j=1}$, where ${\{U_{j,n}\}}_{j=1}^a$ are i.i.d. $\dU[0,1]$ independent of $\pi_n$, and $V_{j,n}:=\pi_n(U_{j,n})$. Let $f_1,\cdots,f_a:[0, 1]^2 \rightarrow \R$ be continuous functions. Now, as in (\ref{eq:proI})
\begin{align}
\E \prod_{j=1}^a f_j(U_{j ,n}, V_{j,n})=&\int_{[0, 1]^a}\E\left(\prod_{j=1}^a f_j(x_j,\pi_n(x_j))\right)\prod_{j=1}^a dx_j=\E \prod_{j=1}^a \tilde{\nu}_{\pi_n}(f_j)+o(1),
\label{eq:uv}
\end{align}
Since $\tilde{\nu}_{\pi_n} \dto \mu$, the RHS of (\ref{eq:uv}) converges to $$\E \prod_{j=1}^a\mu(f_j)=\E_\mu\int_{[0, 1]^s}\prod_{j=1}^af_j(x_j,y_j) d\mu(x_j,y_j).$$
Thus, the joint law of $\{(U_{j ,n}, V_{j,n})\}_{j=1}^a$ converges to $\{(U_{j}, V_{j})\}_{j=1}^a$, where $\{U_j\}_{j=1}^a$ are i.i.d. $\dU[0,1]$ independent of $\mu$, and  given both $\{U_j=u_j\}_{j=1}^n$ and $\mu$,we have $\{V_j\}_{j=1}^a$ are mutually independent with $V_j\sim \mu(\cdot|u_j)$.

Observe that for $g_1, g_2, \ldots, g_a\in \cC[0, 1]$,
\begin{equation}
\E\left(\prod_{j=1}^a g_j(V_{j, n})\Big |U_{1,n}=s_1,\cdots, U_{a,n}=s_a\right)=\E \prod_{j=1}^a g_j(\pi_n(s_j))
\label{eq:1}
\end{equation}
and 
\begin{equation}
\E\left(\prod_{j=1}^a g_j(V_{j})\Big |U_{1}=s_1,\cdots, U_{a}=s_a\right)=\E_\mu\left(\int_{[0, 1]^a}\prod_{j=1}^a g_j(y_j) d\mu(y_j|s_j)\right),
\label{eq:2}
\end{equation}
To show that $\pi_n(\cdot)  \stackrel{w}  \Rightarrow Z(\cdot)$, it suffices to show that the RHS of (\ref{eq:1}) converges to the of RHS (\ref{eq:2}) are equal. Therefore, it suffices to prove that the corresponding conditional distributions converge:
$$\sL({\{V_{j, n}\}}_{j=1}^a| {\{U_{j,n}\}}_{j=1}^a)\dto \sL({\{V_{j}\}}_{j=1}^a| {\{U_{j}\}}_{j=1}^a)$$
This follows by the equicontinuity of $(\pi_n)_{n\geq 1}$ and an application of \cite[Theorem 4]{sweeting}, since $[0,1]^a$ is  compact.
\end{proof}

The following proposition gives a criterion for equicontinuity of permutations. This can be easily checked for many non-uniform models on permutations. To this end, let $\vec x=(x_1, x_2, \cdots, x_s)\in (0, 1]^s$ and define
\begin{equation}
\Omega(i_1, i_2, \cdots i_s, \vec x):=\{\pi_n\in S_n: \pi_n(\lceil nx_1\rceil)=i_1,\cdots,\pi_n(\lceil nx_s\rceil)=i_s\}.
\label{eq:omega}
\end{equation}

\begin{ppn}Let $\pi_n\in S_n$ be a sequence of random permutations. Suppose for all $s\geq 1$ and $1\leq i_1 < i_2 < \cdots < i_s \leq n$, the following holds:
\begin{align}
\lim_{\delta\rightarrow 0}\lim_{n\rightarrow \infty}\sup_{(\vec x, \vec y) \in B(\delta)}\left|\frac{\P(\pi_n\in \Omega(i_1, i_2, \cdots i_s, \vec x))}{\P(\pi_n \in \Omega(i_1, i_2, \cdots i_s, \vec y))}-1\right|=0,
\label{eq:ppn_comparison}
\end{align}
where $B_s(\delta)=\{x_1,\cdots,x_s, y_1,\cdots,y_s\in (0,1]:  |x_s-y_s|\le \delta\}$. Then the permutation sequence $(\pi_n)_{n\geq 1}$ is equicontinuous.
\label{ppn:comp}
\end{ppn}

\begin{proof}
Fix $\varepsilon>0$. Let $\vec x=(x_1, x_2, \cdots, x_s)\in (0, 1]^s$ and $\vec y=(y_1, y_2, \cdots, y_s)\in (0,1]^s$. Denote by $K_{n, s}$ the collection of $s$-tuples $i_1, i_2, \cdots,i_s$ which are all distinct. 

To show equicontinuity, fix $s\geq 1$ and $g_1, g_2, \ldots, g_s\in \cC[0, 1]$. Let $M:=\sup_{x_1, \cdots x_s\in (0, 1]}\prod_{j=1}^sg_j(x_j)$.  Now, using (\ref{eq:ppn_comparison})
\begin{eqnarray}
& &\sup_{\vec x, \vec y \in B_s(\delta)}|G^{(n)}_{g_1, g_2, \ldots, g_s}(\vec x)-G^{(n)}_{g_1, g_2, \ldots, g_s}(\vec y)|\nonumber\\
&\leq & \sup_{\vec x, \vec y \in B_s(\delta)} \sum_{(i_1, i_2, \cdots, i_s) \in K_{n, s}}\left|\prod_{j=1}^sg_j\left(\frac{i_j}{n}\right)\right|\cdot|\P(\pi_n\in \Omega(i_1, i_2, \cdots, i_s, \vec x))-\P(\pi_n\in \Omega(i_1, i_2, \cdots, i_s, \vec y))|\nonumber\\
&\le& M\varepsilon,
\end{eqnarray}
for $n$ large enough and appropriately chosen $\delta$. This implies that the sequence of permutations $(\pi_n)_{n\geq 1}$ is equicontinuous.
\end{proof}

\subsection{Exponential Models on Permutations}

Let $\theta \in \R$, and $T_n: S_n \rightarrow \R$ be any function. Suppose $\pi_n\in S_n$ is a sequence of random permutations with probability mass function
\begin{equation}
\frac{e^{\theta T_n(\sigma)}}{\sum_{\sigma \in S_n}e^{\theta T_n(\sigma)}}.
\label{eq:exp_model}
\end{equation}
One of the most common exponential models on permutations, is the Mallows model, where $T_n(\sigma)$ is the number of inversions of $\sigma$ scaled by $n$. The limiting degree process of the Mallows random permutation is explicitly computed in Section \ref{sec:mallows}. Here, we determine a simple criterion for the convergence of the degree process of a sequence of random permutations distributed as (\ref{eq:exp_model}). To this end, consider the following technical definition:

\begin{defn}For any two fixed vectors  $\vec x=(x_1, x_2, \cdots, x_s)\in (0, 1]^s$ and $\vec y=(y_1, y_2, \cdots, y_s)\in (0, 1]^s$, define the bijection
$$\Phi_{\vec x, \vec y}:\Omega(i_1, i_2, \cdots i_s, \vec x) \rightarrow \Omega(i_1, i_2, \cdots i_s, \vec y)$$ as follows: for each $\pi_n\in \Omega(i_1, i_2, \cdots i_s, \vec x)$, define its image $\tilde \pi_n$ as
\begin{equation}
\tilde \pi_n(k)=
\left\{
\begin{array}{ccc}
i_j =\pi_n(\ceil{n x_j})&  \text{ if } k= \ceil{ny_j}  \text{ for } j\in [s] \\
\pi_n(\ceil{n y_j}  & \text{ if }  k= \ceil{nx_j}  \text{ for } j\in [s]   \\
\pi_n(k)  &     \text{ otherwise.}
\end{array}
\right.
\label{eq:bijection}
\end{equation}
\end{defn}
\noindent Informally, $\Phi_{\vec x, \vec y}$ takes a permutation $\pi_n$ and interchanges the $\ceil{nx_1}, \ceil{nx_2}, \ldots \ceil {nx_s}$ coordinates to the $\ceil{ny_1}, \ceil{ny_2}, \ldots \ceil{ny_s}$ coordinates to get $\tilde \pi_n$. Using this bijection it is easy to get a sufficient condition for the convergence of the permutation process for exponential models.

\begin{cor}
A sequence of random permutations $(\pi_n)_{n\geq 1}$ from (\ref{eq:exp_model}) is equicontinuous whenever
\begin{equation}\lim_{\delta\rightarrow 0}\lim_{n\rightarrow \infty}\sup_{(\vec x, \vec y) \in B(\delta)}\sup_{\pi_n\in \Omega(i_1, \cdots, i_s, \vec x)}|T_n(\pi_n)-T_n(\Phi_{\vec x, \vec y}(\pi_n))|=0,
\label{eq:expcond}
\end{equation}
where $\Phi_{\vec x, \vec y}$ is the bijection  defined in (\ref{eq:bijection}).
\label{cor:expfamily}
\end{cor}

\begin{proof}
Fix $\varepsilon >0$. Let $\tilde \pi_n$ be the image of $\pi_n$ defined by the bijection $\Phi_{\vec x, \vec y}$ in (\ref{eq:bijection}). Then using (\ref{eq:expcond})
\begin{align*}
\left|\frac{\P(\pi_n\in \Omega(i_1, i_2, \cdots, i_s, \vec x))}{\P(\pi_n\in \Omega(i_1, i_2, \cdots, i_s, \vec y))}-1\right|=&\left|\frac{\sum_{\pi_n\in \Omega(i_1, i_2, \cdots, i_s, \vec x)}Q_{n, \theta}(\pi_n)}{\sum_{\tilde \pi_n\in \Omega(i_1, i_2, \cdots, i_s, \vec y)}Q_{n, \theta}(\tilde \pi_n)}-1\right|< \varepsilon,
\end{align*}
for $n$ large enough and appropriately chosen $\delta$. Equicontinuity follows from Proposition \ref{ppn:comp}.
\end{proof}

Consider the following general class of 1-parameter exponential family on the space of permutations $S_n$, with probability mass function
\begin{equation}
Q_{n, f, \theta}(\sigma)=\frac{e^{\theta \sum_{i=1}^nf\left(\frac{i}{n}, \frac{\sigma(i)}{n}\right)}}{\sum_{\sigma \in S_n} e^{\theta \sum_{i=1}^n f\left(\frac{i}{n}, \frac{\sigma(i)}{n}\right)}},
\label{eq:exp_f}
\end{equation}
where $f:[0, 1]^2\rightarrow [0, 1]$ is any continuous function.  This is a special case of the model (\ref{eq:exp_model}) with $T_n(\sigma)=\sum_{i=1}^nf\left(\frac{i}{n}, \frac{\sigma(i)}{n}\right)$. Popular choices of the function $f$ includes the {\it Spearman's Rank Correlation Model}: $f(x,y)=-(x-y)^2$ and the {\it Spearman's Footrule Model}:  $f(x,y)=-|x-y|$. These models find applications in statistics for analyzing ranked data \cite{diaconis_group,perm_statistics}. Feigin and Cohen \cite{feigin_cohen} gave a nice application of such models for analyzing agreement between several judges in a contest. For other choices of $f$ and their various properties refer to Diaconis \cite{diaconis_group}. Consistent estimation of parameters in such models has been studied recently by Mukherjee \cite{sm}.

Using Corollary \ref{cor:expfamily} it can easily shown that any sequence of random permutations $(\pi_n)_{n\geq 1}$ distributed as (\ref{eq:exp_f}) is equicontinuous, that is, their corresponding degree process converges:

\begin{cor}Fix $\theta \in \R$ and $f:[0, 1]^2\rightarrow [0, 1]$ be any continuous function. Let $\pi\in S_n$ be a sequence of random permutations distributed as (\ref{eq:exp_f}). Then $(\pi_n)_{n\geq 1}$ is equicontinuous and the degree process $d_n(\cdot)$ converges.
\label{cor:f}
\end{cor}

\begin{proof} The convergence of the degree process follows from the equicontinuity (Proposition \ref{new_results}). To prove equicontinuity, let $T_n(\sigma)=\sum_{i=1}^n f(i/n,\sigma(i)/n)$. Fix $\delta>0$ and let $(\vec x, \vec y) \in B(\delta)$. Using the bijection (\ref{eq:bijection}) we get
\begin{align}
|T_n(\pi_n)-T_n(\tilde \pi_n)|&\le 2\sum_{j=1}^s\left|f\left(\frac{\ceil{n x_j}}{n}, \frac{\pi_n(\ceil{n x_j})}{n}\right)-f\left(\frac{\ceil{n y_j}}{n}, \frac{\pi_n(\ceil{n x_j})}{n}\right)\right|\nonumber\\
&\leq 2s\sup_{\substack{|x_1-x_2|\le \delta+\frac{1}{n},\\ y\in [0,1]}}|f(x_1,y)-f(x_2,y)|,
\end{align}
which goes to 0, after taking $n\rightarrow \infty$ and $\delta\rightarrow 0$, by the continuity of $f$.
\end{proof}

\section{Mallows Random Permutation}
\label{sec:mallows}

For $\beta\in\R$ and $n\geq 1$, the Mallows measure over permutations $\pi_n\in S_n$ is given by the probability mass function
$$m_{\beta, n}(\sigma):=\frac{e^{-\beta\cdot\frac{\lambda(\sigma)}{n}}}{\sum_{\sigma\in S_n}e^{-\beta\cdot\frac{\lambda(\sigma)}{n}}},$$
where $\lambda(\sigma)=|\{(i, j): (i-j)(\pi_n^{-1}(i)-\pi_n^{-1}(j))<0\}|$ is the number of inversions of the permutation $\sigma$. Diaconis and Ram \cite{diaconis_ram} studied a Markov chain on $S_n$ for which the Mallows model gives the limiting distribution.  
Tail bounds for the displacement of an element in a Mallows permutation was studied by Braverman and Mossel \cite{baverman_mossel}. Recently,  Mueller and Starr \cite{mueller_starr} and later Bhatnagar and Peled \cite{bhatnagar_peled} studied the length of the longest increasing subsequence in a Mallows permutation. Consistent estimation of parameters in exponential models on permutations has been studied by Mukherjee \cite{sm}.

The uniform random permutation corresponds to the case $\beta=0$. It is well known that for $\pi_n\in S_n$ chosen uniformly at random, $\nu_{\pi_n}$ converges weakly in probability to $\dU(0, 1)$. This was generalized to Mallow random permutations by Starr \cite{starr}.

\begin{thm}[Starr \cite{starr}] 
Let $\pi_n\in S_n$ be a Mallows random permutation with parameter $\beta$. Then the empirical permutation measure $\tilde \nu_{\pi_n}$ converges weakly in probability to a random variable which has density in $[0, 1]^2$ given by
\begin{equation}
m_\beta(x, y):=\frac{\beta  \sinh \left(\frac{\beta }{2}\right)}{2 \left(\exp \left(\frac{\beta }{4}\right) \cosh \left(\frac{1}{2} \beta  (x-y)\right)-\exp \left(-\frac{\beta }{4}\right) \cosh \left(\frac{1}{2} \beta  (x+y-1)\right)\right)^2},
\label{eq:density_mallows}
\end{equation}
and distribution function 
\begin{equation}
M_{\beta}(a, b)=-\frac{1}{\beta} \log \left(\frac{2 \exp \left(-\frac{1}{2}\beta(a+b-1)\right)\left(\sinh \left(\frac{a \beta }{2}\right) \sinh \left(\frac{\beta  b}{2}\right)\right)}{\sinh \left(\frac{\beta }{2}\right)}-1\right).
\label{eq:distribution_mallows}
\end{equation}
\label{th:starr}
\end{thm}

\begin{proof}The proof of (\ref{eq:density_mallows}) can be found in Starr \cite{starr}. The expression for the distribution function (\ref{eq:distribution_mallows}) follows by directly integrating the density $m_\beta$ in (\ref{eq:density_mallows}).
\end{proof}

By the above theorem, for $\beta \in \R$ and a sequence of permutations $\pi_n\sim M_{\beta, n}$, the empirical permutation measure $\tilde \nu_{\pi_n}\dto M_{\beta}$. This together with Theorem \ref{thm:degreejoint} can be used to compute the limiting density of the degree proportion in a Mallows random permutation. 


\subsection{Proof of Theorem \ref{th:mallowsdegreejoint}}

The proof of Theorem \ref{th:mallowsdegreejoint} has two parts: to show the existence of limit of the degree process $d_n(\cdot)$ by verifying  (\ref{eq:expcond}) in Corollary \ref{cor:expfamily}, and the explicit computation of the density of the limiting distribution using Theorem \ref{th:starr}.

\subsubsection{Existence of the Limit}

In this section we show that the degree process of a Mallows random permutation converges. In light of Proposition \ref{new_results} it suffices to verify that the Mallows random permutation is equicontinuous:

\begin{lem}Let $\beta \in \R$ and $\pi_n\sim M_{\beta, n}$ be a sequence Mallows random permutations. Then $(\pi_n)_{n\geq 1}$ is equicontinuous.
\end{lem}

\begin{proof} Fix $\varepsilon>0$ and $1\leq i_1<i_2\cdots i_s\leq n$. Let $\vec x=(x_1, x_2, \cdots, x_s)\in (0, 1]^s$ and $\vec y=(y_1, y_2, \cdots, y_s)\in (0,1]^s$ and consider the bijection (\ref{eq:bijection}) between $\Omega(i_1, i_2, \ldots, i_s, \vec x)$ and $\Omega(i_1, i_2, \ldots, i_s, \vec y)$. If $\tilde \pi_n$ denotes the image of $\pi_n$ under this bijection, then $$\frac{1}{n}|\lambda(\pi_n)-\lambda (\tilde \pi_n)|\le \frac{1}{n}\sum_{a=1}^s|\lceil nx_a\rceil -\lceil ny_a\rceil|\le s \delta +\frac{2s}{n},$$
which goes to zero after taking limits as $n\rightarrow\infty$ and $\delta\rightarrow 0$.  Equicontinuity of $(\pi_n)_{n\geq 1}$ now follows from Corollary \ref{cor:expfamily}.
\end{proof}

The above result and Proposition \ref{new_results} implies that the permutation process $\pi_n(\cdot)\stackrel{w}\Rightarrow W_\beta(\cdot)$, such that for every $t\geq 0$, $W_\beta(t)$ is independent and distributed according to conditional law of $Q_1|Q_2=t$, where $(Q_1, Q_2)\sim M_\beta$.  Since the distribution of $M_\beta$ of $(Q_1, Q_2)$ has uniform marginals,  
\begin{equation}
\P(W_\beta(t)\leq w)=\int_{0}^w m_{\beta}(x, t)dx,
\label{eq:W}
\end{equation}
and $W_\beta(t)$ has density $m_{\beta}(\cdot, t)$. Theorem \ref{thm:degreejoint} then implies that $d_n(\cdot)\stackrel{w}\Rightarrow D_\beta(\cdot)$, where $$D_\beta(t)=t+W_\beta(t)-2M_\beta(t, W_\beta(t)),$$
and $D_\beta(t)$ is independent for all $t\geq 0$. Therefore, for indices $0\leq  r_1<r_2< \cdots <r_s\leq 1$,
$$\left(\frac{d_n(\ceil{nr_1})}{n}, \frac{d_n(\ceil{nr_2})}{n}, \ldots, \frac{d_n(\ceil{nr_s})}{n}\right) \stackrel{\sD}{\rightarrow}(D_{1, \beta}, D_{2, \beta}, \ldots D_{s, \beta}),$$
where $D_{a, \beta}=D_\beta(a)$, for $a\in [0,1]$. 

\subsubsection{Calculating the Limiting Density} Fix $\beta \in \R$ and $a\in [0, 1]$ and suppose $W\sim W_\beta(a)$ be distributed as in (\ref{eq:W}). To find the density of $D_{a, \beta}$ for $a\in [0,1]$, we have to find the density of the random variable
\begin{eqnarray}
J_{a, \beta}(W)&:=&a+W-2 M_\beta(a, W)\nonumber\\
&=&a+W+\frac{2}{\beta} \log \left(\frac{2 \exp \left(-\frac{1}{2}\beta(a+W-1)\right)\left(\sinh \left(\frac{\beta a}{2}\right) \sinh \left(\frac{\beta  W}{2}\right)\right)}{\sinh \left(\frac{\beta }{2}\right)}-1\right) \nonumber\\
&=&a+W+\frac{2}{\beta} \log \left(\varphi_\beta(a)\left(1-e^{-\beta W}\right)-1\right)\nonumber, 
\label{eq:J1}
\end{eqnarray}
where $$\varphi_\beta(a):=e^{\frac{1}{2}(\beta-\beta a)} \text{csch}\left(\frac{\beta }{2}\right) \sinh \left(\frac{a \beta }{2}\right).$$

We begin by establishing properties of the function $J_{a,\beta}: \R \to\R$ defined as $J_{a,\beta}(w)=a+w-2 M_\beta(a, w)$, for $a\in [0, 1]$. Recall that an interval $[a, b]$ is always interpreted as $[a\vee b, a\wedge b]$

\begin{lem}Let $\beta\geq 0$, $a\in [0, 1]$, and $a_c(\beta)$ be as defined in Theorem \ref{th:mallowsdegreejoint}.  Then for $J_{a,\beta}$ as defined above, the following hold:
\begin{itemize}
\item[{\it (a)}]The function $J_{a,\beta}$ is strictly convex in $\R$ for $\beta> 0$ and linear for $\beta=0$.

\item[{\it (b)}]For $a\in [0, a_c(\beta)]$ the function $J_{a,\beta}$ is strictly increasing and for $a\in [1-a_c(\beta), 1]$, the function $J_{a,\beta}$ is strictly decreasing in $[0, 1]$. 

\item[{\it(c)}]For $a\in (a_c(\beta), 1-a_c(\beta))$, the function $J_{a,\beta}$ has a minimum at $z_0\in (0, 1)$, and is strictly decreasing in $[0, z_0)$ and strictly increasing in $(z_0, 1]$.

\end{itemize}
\label{eq:jprop}
\end{lem}

\begin{proof}The derivatives of the function $J_{a,\beta}$ are  
$$J_{a,\beta}'(z)=\frac{d}{dz}J_{a,\beta}(z)=1+\frac{2\varphi_\beta(a)e^{-\beta z}}{\varphi_\beta(a)\left(1-e^{-\beta z}\right)-1}$$
and 
$$J_{a,\beta}''(z)=2\varphi_\beta(a)\cdot\frac{-\beta e^{-\beta z}\left(\varphi_\beta(a)\left(1-e^{-\beta z}\right)-1\right)-\beta e^{-2 \beta z}\varphi_\beta(a)}{\left(\varphi_\beta(a)(\left(1-e^{-\beta z}\right)-1\right))^2}=-\frac{\beta e^{\beta z}\left(\varphi_\beta(a)-1\right)}{\varphi_\beta(a)\left(\left(1-e^{-\beta z}\right)-1\right)^2}.$$
Note that $\varphi_a(\beta)-1=\frac{1-e^{\beta -a \beta }}{e^{\beta }-1}\leq 0$, for all $\beta \in \R$. Therefore, for $\beta >0$, $J_{a,\beta}''(z)>0$ for all $z\in\R$, and so, $J_{a,\beta}$ is a convex function. 


The convexity of $J_{a,\beta}$ implies that $J_{a,\beta}'$ is increasing, and $J_{a,\beta}'(z)=0$ has at most one solution $z_0$ in $[0, 1]$:
$$z_0:=\frac{1}{\beta}\log\left(\frac{\varphi_\beta(a)}{1-\varphi_\beta(a)}\right)\in (0, 1) \Leftrightarrow a\in [a_c(\beta), 1-a_c(\beta)],$$
where $a_c(\beta)$ is defined in Theorem \ref{th:mallowsdegreejoint}. Therefore, for  $a\notin [a_c(\beta), 1-a_c(\beta)]$, the function $J_{a,\beta}$ is strictly monotone: for $a\in [0, a_c(\beta)]$ the function $J_{a,\beta}$ is strictly increasing, and for $a\in [1-a_c(\beta), 1]$, the function $J_{a,\beta}$ is strictly decreasing in $[0, 1]$.

For $a\in [a_c(\beta), 1-a_c(\beta)]$, the function $J_{a,\beta}$ has a minimum at $z_0$, and is strictly decreasing in $[0, z_0)$ and strictly increasing in $(z_0, 1]$.
\end{proof}

Recall the definition of $R(a, \beta)$ from (\ref{eq:Rab}). For notational convenience, for $\beta>0$ and $a\in [0, 1]$ define the function $h_{a, \beta}:[0, 1]\to \R$,
\begin{eqnarray}
h_{a, \beta}(x)&=&\frac{\beta  e^{\frac{1}{2} \beta  (a-x)}}{(1-e^{-\beta})\sqrt{e^{\beta  (a+x)}-e^\beta R(a, \beta)}}\nonumber\\
&=&\frac{\beta  e^{\frac{1}{2} \beta  (a-x+2)}}{\sqrt{4e^\beta(e^{\beta a}-1) (e^{\beta a}-e^{\beta})
+e^{\beta  (a+x)}(1-e^\beta)^2}}.
\end{eqnarray}

The above lemma will now be used to complete the proof of Theorem \ref{th:mallowsdegreejoint}. 
Assume that $\beta\geq 0$, and consider the two cases (recall the definition of $J_{a, \beta}$ from (\ref{eq:J1})):

\begin{description}

\item[1] Suppose $a\notin [a_c(\beta), 1-a_c(\beta)]$. In this case,  $J_{a,\beta}$ is increasing in $[0,1]$ (Lemma \ref{eq:jprop}) and the equation $J_{a,\beta}(z)=w$ has a unique solution
$J_{a, \beta}^{-1}(w)\in [0, 1]$. Then by the Jacobian transformation and direct calculations, the density of $J_{a,\beta}(W)$ simplifies to 
$$g_{a, \beta}(w)=\left|\frac{d}{d w}J_{a, \beta}^{-1}(w)\right|f_\beta(J_{a, \beta}^{-1}(w), a)=h_{a, \beta}(w).$$
The support of $J_{a,\beta}(W)$ is $[J_{a,\beta}(0), J_{a,\beta}(1)]=[a, 1-a]$.

\item[2] Suppose $a\in [a_c(\beta), 1-a_c(\beta)]$.  In this case, for $w\in [0, 1]$ the equation $J_{a,\beta}(z)=w$ has at most two solutions in  $[0, 1]$ depending on the value of $w$ (Lemma \ref{eq:jprop}). 
\begin{description}
\item[2.1]$w \in [J_{a,\beta}(0), J_{a,\beta}(1)]=[a, 1-a]$. This situation is same as the previous case, that is, $J_{a,\beta}(z)=w$ has a unique solution $J_{a, \beta}^{-1}(w)\in [0, 1]$, and the density of $J_{a,\beta}(W)$ simplifies to 
$$g_{a, \beta}(w)=\left|\frac{d}{d w}J_{a, \beta}^{-1}(w)\right|f_\beta(J_{a, \beta}^{-1}(w), a)=h_{a, \beta}(w).$$
for $w\in[a, 1-a]$.

\item[2.2]$w \notin [J_{a,\beta}(0), J_{a,\beta}(1)]=[a, 1-a]$. In this case, $J_{a,\beta}(z)=w$ has two solutions  given by $J_{a, \beta, 1}^{-1}(w)=-\frac{2}{\beta}\log \gamma_1(w)$ and $J_{a, \beta, 2}^{-1}(w)=-\frac{2}{\beta}\log\gamma_2(w)$, where $\gamma_1(w)$ and $\gamma_2(w)$ are roots of the quadratic 
$$\gamma e^{\frac{\beta}{2}(w-a)}+1= \varphi_\beta(a)\left(1-\gamma^2\right)\Rightarrow \varphi_\beta(a)\gamma^2+\gamma e^{\frac{\beta}{2}(w-a)}+1-\varphi_\beta(a)=0.$$
The above quadratic equation is obtained by simplifying the equation $J_{a,\beta}(z)=w$ and substituting $\gamma=e^{-\beta z/2}$.
Then by the Jacobian transformation, the density of $J_{a,\beta}(Z)$ is 
\begin{equation}
g_{a, \beta}(w)=\sum_{s=1}^{2}\left|\frac{d}{d w}J_{a, \beta, s}^{-1}(w)\right|f_\beta(J_{a, \beta, s}^{-1}(w), a)
\label{eq:jacobian_I}
\end{equation}
Substituting $J_{a, \beta, 1}^{-1}(w), J_{a, \beta, 2}^{-1}(w)$, and the density $f_\beta(\cdot, a)$ (\ref{eq:density_mallows}), and simplifying (\ref{eq:jacobian_I}) gives $g_{a, \beta}(w)=2h_{a, \beta}(w)$. Since the function $J_{a,\beta}$ has a minimum at $z_0$, and is strictly decreasing in $[0, z_0)$ and strictly increasing in $(z_0, 1]$, the support of $J_{a,\beta}(W)$ is $$[J_{a,\beta}(z_0), J_{a,\beta}(0)\vee J_{a,\beta}(1)]= \left[1-a+\frac{1}{\beta}\log R(a, \beta), a\vee 1-a\right].$$
\end{description}
\end{description}

For the $\beta<0$ the result can be proved similarly. Theorem \ref{th:mallowsdegreejoint} holds verbatim even for $\beta<0$, if every interval $[a, b]$ is interpreted as $[a\vee b, a\wedge b]$.

\section{Asymptotics For The Minimum Degree: Proof of Theorem \ref{min_deg}}
\label{sec:min_deg}

This section gives the proof of the limiting Rayleigh distribution of the minimum degree  in a uniformly random permutation graph.

For $i \in [n]$ define 
\begin{equation}
c_n(i)=
\left\{
\begin{array}{ccc}
i+\pi_n(i)  & \text{ for } 1\leq  i < \frac{n+1}{2}    \\
2(n+1)-i-\pi_n(i)  &   \text{ for }  \frac{n+1}{2} <  i \leq n.
\end{array}
\right.
\label{cndefn}
\end{equation}

The following lemma shows that the degrees $d_n(i)$ can be small (order $\sqrt{n}$) only if $c_n(i)$ is small, which can happen only if  $i$ is such that either $i$ or $n+1-i$ is small (order $\sqrt{n}$).

\begin{lem}\label{lem_hyper1}
For any $\gamma\in (0,\infty)$ 
\begin{align*}
\lim_{M\rightarrow\infty}\lim_{n\rightarrow\infty}\sum_{i=1}^n\P(d_n(i)\le \gamma\sqrt{n},c_n(i)>M\sqrt{n})=0.
\end{align*}
\end{lem}

\begin{proof}
By symmetry it suffices to show that
$$\lim_{M\rightarrow\infty}\lim_{n\rightarrow\infty}\sum_{1\le i\le \frac{n+1}{2}}\P(d_n(i)\le \gamma \sqrt{n},i+\pi_n(i)>M\sqrt{n})=0,$$
which follows if we can show the following:
\begin{equation}
\label{hyper4}
\lim_{n\rightarrow\infty}\sum_{ \frac{n+1}{4}\le i\le \frac{n+1}{2}}\P(d_n(i)\le\gamma \sqrt{n})=0,
\end{equation}
\begin{equation}\label{hyper5}
\lim_{M\rightarrow\infty}\lim_{n\rightarrow\infty}\sum_{\frac{M\sqrt{n}}{2}\le i\le  \frac{n+1}{4}}\P(d_n(i)\le\gamma \sqrt{n})=0,
\end{equation}
\begin{equation}
\label{hyper6}
\lim_{M\rightarrow\infty}\lim_{n\rightarrow\infty}\sum_{1\le i\le \frac{M\sqrt{n}}{2}}\sum_{\frac{M\sqrt{n}}{2}\le j\le n}\P(d_n(i)\le \gamma\sqrt{n},\pi_n(i)=j)=0. 
\end{equation}

Recall by Proposition \ref{hyper}, $\E(d_n(i)|\pi_n(i)=j)=\frac{(i-1)(n-j)+(j-1)(n-i)}{n-1}=\frac{j(n-2i+1)+(n+1)i-2n}{n-1}$. Therefore, for $\frac{n+1}{4}\le i \le \frac{n+1}{2}$,
 \begin{align*}
\E(d_n(i)|\pi_n(i)=j)-\gamma\sqrt{n}\ge j\left(\frac{n+1}{2(n-1)}\right)+i\left(\frac{n+1}{n-1}\right)-(\gamma\sqrt{n}+2)\ge\frac{n+1}{8},
\end{align*}
for all $n$ large enough. An application of Lemma \ref{hyper} now gives 
$$\P(d_n(i)\le \gamma\sqrt{n}|\pi_n(i)=j)\le 2e^{-\frac{n}{128}}.$$
On adding over $i$ and $j$ gives
$$\sum_{\frac{n+1}{4}\le i\le\frac{n+1}{2}}\P(d_n(i)\le \gamma\sqrt{n},i+\pi_n(i)>M\sqrt{n})\le \sum_{\frac{n+1}{4}\le i\le \frac{n+1}{2}}\P(d_n(i)\le \gamma\sqrt{n})\le ne^{-\frac{n}{128}},$$
from which (\ref{hyper4}) follows. 

Proceeding to prove (\ref{hyper5}), for $\frac{n+1}{4}\le i \le \frac{n+1}{2}$,
 \begin{align*}
\E(d_n(i)|\pi_n(i)=j)-\gamma\sqrt{n}\geq &j\left(\frac{n+1}{2(n-1)}\right)+i-(\gamma\sqrt{n}+2)\ge\frac{i+j}{2}
\end{align*}
for all $M\ge 4\gamma+8$. Lemma \ref{hyper} gives 
$$\P(d_n(i)\le \gamma\sqrt{n}|\pi_n(i)=j)\le 2\sum_{j\ge 1}e^{-(i^2+j^2)/8},$$
which on summing over $i$ and $j$ gives
$$\sum_{\frac{M\sqrt{n}}{2}\le i\le \frac{n+1}{4}}\P(d_n(i)\le \gamma\sqrt{n})\le 2\int_{M/2}^\infty e^{-x^2/8}dx\int_0^\infty e^{-y^2/8}dy.$$
Since the RHS of the above equation  goes to $0$ on letting $M\rightarrow\infty$, (\ref{hyper5}) follows.

Finally, to show (\ref{hyper6}), for $1\le i\le \frac{M\sqrt{n}}{2}$, and $\frac{M\sqrt{n}}{2} \le j\le n$, note that
\begin{align*}
\E(d_n(i)|\pi_n(i)=j)-\gamma\sqrt{n}\ge j/2-\gamma\sqrt{n}-2+i\geq \frac{i+j}{4},
\end{align*}
for $M\ge 4\gamma+8$. Thus by a similar argument as before, we have
$$\lim_{n\rightarrow\infty}\sum_{1\le i\le \frac{M\sqrt{n}}{2}}\sum_{\frac{M\sqrt{n}}{2}\le j\le n}\P(d_n(i)\le \gamma\sqrt{n},\pi_n(i)=j)\le 
2\int_{0}^{M/2} e^{-x^2/32}dx\int_{M/2}^\infty e^{-y^2/32}dy,$$
 which goes to $0$ as $M\rightarrow\infty$ as before. This completes the proof of the lemma.
\end{proof}

The following lemma now strengthens the above result to show that  $d_n(i)$ and  $c_n(i)$ are close for those indices $i$ where either $i$ or $n+1-i$ is small.

\begin{lem}\label{lem_hyper2}
For any $\gamma>\varepsilon>0$
\begin{itemize}
\item[{\it (a)}] $\lim_{n\rightarrow\infty}\sum_{i=1}^n\P(d_n(i)\le \gamma\sqrt{n},c_n(i)> (\gamma+\varepsilon)\sqrt{n})=0$,
\item[{\it (b)}] $\lim_{n\rightarrow\infty}\sum_{i=1}^n\P(d_n(i)>\gamma\sqrt{n},c_n(i)\le (\gamma-\varepsilon)\sqrt{n})=0$.
\end{itemize}
\end{lem}

\begin{proof}
We claim that for every fixed $M<\infty,\varepsilon>0$,
\begin{align}\label{hyper7}
\lim_{n\rightarrow\infty}\sum_{1\le i\le n}\P(|d_n(i)-c_n(i)|>\varepsilon\sqrt{n},c_n(i)\le M\sqrt{n})=0.
\end{align}

Proceeding to complete the proof of the lemma using \eqref{hyper7}, note that
\begin{align*}
&\P(d_n(i)\le \gamma\sqrt{n}, c_n(i)>(\gamma+\varepsilon)\sqrt{n})\\
\le& \P(d_n(i)\le \gamma\sqrt{n}, c_n(i)>M\sqrt{n})+\P(|d_n(i)-c_n(i)|> \varepsilon\sqrt{n},c_n(i)\le M\sqrt{n}).
\end{align*}
Summing over $i$ and letting $n\rightarrow\infty$ followed by $M\rightarrow\infty$, the second term goes to $0$ by (\ref{hyper7}), and the first term goes to $0$ by Lemma \ref{lem_hyper1}. This completes the proof of part (a). The proof of part (b) follows by similar calculations.

Turning to the proof of (\ref{hyper7}), note that by symmetry it suffices to show that
\begin{align*}
\lim_{n\rightarrow\infty}\sum_{1\le i\le \frac{n+1}{2}}\P(|d_n(i)-c_n(i)|>\varepsilon\sqrt{n},c_n(i)\le M\sqrt{n})=0,
\end{align*}
which follows if we can show
\begin{align*}
\lim_{n\rightarrow\infty}\frac{1}{n} \sum_{1\le i\le M\sqrt{n}}\sum_{1\le j\le M\sqrt{n}}\P(|d_n(i)-i-j|>\varepsilon\sqrt{n}|\pi_n(i)=j)=0.
\end{align*}

To this end, for $1\le i,j\le M\sqrt{n}$,
$$\left|\E(d_n(i)|\pi_n(i)=j)-i-j\right|=\left|\frac{(i-1)(n-j)+(j-1)(n-i)}{n-1}-i-j\right|\le  2M,$$
and
$$
\Var(d_n(i)|\pi_n(i)=j)=\frac{(i-1)(j-1)(n-i)(n-j)}{(n-1)^2(n-2)}\le 2M.$$
Therefore, by Chebyshev's inequality
\begin{align*}
\P(|d_n(i)-i-j|\ge \varepsilon\sqrt{n}|\pi_n(i)=j)\le \frac{4M^2}{(\varepsilon\sqrt{n}-2M)^2}.
\end{align*}
This readily gives
\begin{align*}
\frac{1}{n} \sum_{1\le i\le M\sqrt{n}}\sum_{1\le j\le M\sqrt{n}}\P(|d_n(i)-i-j|>\varepsilon\sqrt{n}|\pi_n(i)=j)\le \frac{4M^4}{(\varepsilon\sqrt{n}-2M)^2},
\end{align*}
which goes to $0$ on letting $n\rightarrow\infty$, for every $M<\infty$ and $\varepsilon>0$.
\end{proof}

\subsubsection{Completing the proof of Lemma \ref{min_deg}}

Using the above lemmas we can now complete the proof of the theorem. To this end, it suffices to show that for any $\gamma>0$
\begin{align}
\lim_{n\rightarrow\infty}\P(d_n(i)>\gamma\sqrt{n}, \text{ for all }1\le i\le n)=e^{-\gamma^2/2}.
\label{eq:gamma}
\end{align}
Note that
\begin{align*}
\P(d_n(i)>\gamma\sqrt{n},1\le i\le n)\le \sum_{i=1}^n\P(d_n(i)>\gamma\sqrt{n},c_n(i)\le (\gamma-\varepsilon)\sqrt{n})+\P(c_n(i)>(\gamma-\varepsilon)\sqrt{n},1\le i\le n),
\end{align*}
and so by Lemma \ref{lem_hyper2}
$$\limsup_{n\rightarrow\infty}\P(d_n(i)>\gamma\sqrt{n},1\le i\le n)\le \limsup_{n\rightarrow\infty}\P(c_n(i)>(\gamma-\epsilon)\sqrt{n},1\le i\le n).$$
A similar argument gives 
$$\limsup_{n\rightarrow\infty}\P(c_n(i)>\gamma\sqrt{n},1\le i\le n)\limsup_{n\rightarrow\infty}\P(d_n(i)>(\gamma-\epsilon)\sqrt{n},1\le i\le n),$$
and so
 to prove (\ref{eq:gamma}) and hence the theorem, it suffices to prove the following lemma:

\begin{lem}\label{lem:limit_dist} Let $c_n(\cdot)$ be as defined in (\ref{cndefn}). Then  for any $\gamma>0$, 
\begin{align}
\lim_{n\rightarrow\infty}\P\left(\min_{1\le i\le n}c_n(i)> \gamma\sqrt{n}\right)=e^{-\gamma^2}.
\label{eq:cn}
\end{align}
\end{lem}

\begin{proof}
Note that 
\begin{align}
\P\left(\min_{1\le i\le \gamma\sqrt{n}}c_n(i)>\gamma\sqrt{n}\right)=&\P\left(\pi_n(j)>\gamma\sqrt{n}-j, 1\leq j \leq \gamma\sqrt n\right)\nonumber\\
=&\frac{(n-\floor{\gamma \sqrt n}+1)^{\floor{\gamma n}}}{n(n-1)\cdots (n-\floor{\gamma n}+1)}\nonumber\\
=&\frac{(n-\floor{\gamma \sqrt n}+1)^{\floor{\gamma n}}(n-\floor{\gamma n})!}{n!}.
\label{align2}
\end{align}		
Moreover, 
\begin{align}
& \P\left(\min_{n+1-\gamma\sqrt{n}\le j\le n}c_n(j)>\gamma\sqrt{n}\Big |\pi_n(i),1\le i\le \gamma\sqrt{n}\right)\nonumber\\
=&\P\left(\pi_n(j) < 2(n+1)-j-\gamma\sqrt{n},n+1-\gamma\sqrt{n}\le j\le n|\pi_n(i),1\le i\le \gamma\sqrt{n}\right)\nonumber\\
\ge&\frac{(n-1-2\floor{\gamma \sqrt n})^{\floor{\gamma \sqrt n}}(n-2\floor{\gamma \sqrt n})!}{(n-\floor{\gamma \sqrt n})!},
\label{align1}
\end{align}
where the lower bound uses the following argument: The probability of the event is minimized when all the $\pi_n(i)$, for $i \in [\floor{\gamma\sqrt n}]$, are at most $n+1-\gamma \sqrt n$. This minimizes the choices of $\pi_n(j)$, for $n+1-\gamma\sqrt{n}\le j\le n$. In this case, each $\pi_n(j)$ has $(n-1-2\floor{\gamma \sqrt n})$ choices, and the bound follows.

Combining (\ref{align1}) and (\ref{align2}) and taking limits as $n\rightarrow \infty$ gives the lower bound
\begin{align}
 &\P\left(\min_{1\le i\le n}c_n(i)>\gamma\sqrt{n}\right)\nonumber\\
=&\E \left(\P\left(\min_{n+1-\gamma\sqrt{n}\le j\le n}c_n(j)>\gamma\sqrt{n}\Big |\pi_n(i),1\le i\le \gamma\sqrt{n}\right)\pmb 1\left\{\min_{1\le i\le \gamma\sqrt{n}}c_n(i)>\gamma\sqrt{n}\right\}\right)\nonumber\\
\ge & \frac{(n-\floor{\gamma \sqrt n}+1)^{\floor{\gamma n}}(n-1-2\floor{\gamma \sqrt n})^{\floor{\gamma \sqrt n}} (n-\floor{\gamma n})! (n-2\floor{\gamma \sqrt n})!}{n! (n-\floor{\gamma \sqrt n})!}\nonumber\\
&\rightarrow e^{-\gamma^2}.
\label{eq:mindeglb}
\end{align}

For the upper bound, setting $N_n:=|\{1\le i\le \gamma\sqrt{n}:\pi_n(i)\ge  n+1-\gamma\sqrt{n}\}|$ and fixing a large integer $M$  we have
\begin{align}
\P(\min_{1\le i\le n}c_n(i)>\gamma\sqrt{n}) \le & \P(\min_{1\le i\le n}c_n(i)>\gamma\sqrt{n},N_n\le M)+\P(N_n>M)\nonumber\\
\le & \P(\min_{1\le i\le n}c_n(i)>\gamma\sqrt{n},N_n\le M)+\frac{\E N_n}{M},
\label{align3}
\end{align}
by Markov's inequality. Now, since 
$$\E N_n=\sum_{i=1}^{\lfloor \gamma\sqrt{n}\rfloor}\frac{\lfloor \gamma\sqrt{n}\rfloor}{n}\le \gamma^2,$$
the second term in the RHS of (\ref{align3}) to $0$ after taking limits as $n\rightarrow\infty$ followed by $M\rightarrow\infty$. Again, by a similar argument as the lower bound,  on the set $\{N_n\le M\}$ we have
\begin{align}
& \P\left(\min_{n+1-\gamma\sqrt{n}\le j\le n}c_n(j)>\gamma\sqrt{n}|\pi_n(i),1\le i\le \gamma\sqrt{n}\right)\nonumber\\
=&\P\left(\pi_n(j)<2(n+1)-j-\gamma\sqrt{n},n+1-\gamma\sqrt{n}\le j\le n|\pi_n(i),1\le i\le \gamma\sqrt{n}\right)\nonumber\\
\leq &\frac{(n-1-2\floor{\gamma \sqrt n}+M)^{\floor{\gamma \sqrt n}}(n-2\floor{\gamma \sqrt n}+M)!}{(n-\floor{\gamma \sqrt n})!}.
\label{eq:degub}
\end{align}
Therefore, using (\ref{align2}), (\ref{align3}) and (\ref{eq:degub}),
\begin{align}
&\P\left(\min_{1\le i\le n}c_n(i)>\gamma\sqrt{n},N_n\le M\right)\nonumber\\
=&\E \left(\P\left(\min_{n+1-\gamma\sqrt{n}\le j\le n}c_n(j)>\gamma\sqrt{n}|\pi_n(i),1\le i\le \gamma\sqrt{n}\right)\pmb 1\{\min_{1\le i\le \gamma\sqrt{n}}c_n(i)>\gamma\sqrt{n}\} \pmb 1\left\{N_n\le M\right\}\right)\nonumber\\
\leq & \frac{(n-1-2\floor{\gamma \sqrt n}+M)^{\floor{\gamma \sqrt n}}(n-2\floor{\gamma \sqrt n}+M)!}{(n-\floor{\gamma \sqrt n})!}\P\left(\min_{1\le i\le \gamma\sqrt{n}}c_n(i)>\gamma\sqrt{n}\right)\nonumber\\
\rightarrow& e^{-\gamma^2},
\label{eq:mindegub}
\end{align}
by taking limits as $n\rightarrow \infty$ and $M\rightarrow \infty$. This completes the proof of the upper bound, which combined with the lower bound (\ref{eq:mindeglb}) gives the result.
\end{proof}

\end{document}